\documentclass[american,english]{article}
\usepackage[T1]{fontenc}
\usepackage[utf8]{inputenc}
\usepackage{babel}
\usepackage{amsmath}
\usepackage{amsthm}
\usepackage{amssymb}
\usepackage{graphicx}

\makeatletter
\numberwithin{equation}{section}
\numberwithin{figure}{section}
\theoremstyle{plain}
\newtheorem{thm}{\protect\theoremname}
  \theoremstyle{definition}
  \newtheorem{defn}[thm]{\protect\definitionname}
  \theoremstyle{plain}
  \newtheorem{prop}[thm]{\protect\propositionname}
  \theoremstyle{definition}
  \newtheorem{example}[thm]{\protect\examplename}
  \theoremstyle{remark}
  \newtheorem{rem}[thm]{\protect\remarkname}
  \theoremstyle{plain}
  \newtheorem{cor}[thm]{\protect\corollaryname}
  \theoremstyle{plain}
  \newtheorem{lem}[thm]{\protect\lemmaname}
  \theoremstyle{plain}
  \newtheorem{fact}[thm]{\protect\factname}

\usepackage{ae,aecompl} 

\usepackage{fullpage} 
\usepackage{enumitem}

\renewenvironment{enumerate}{\begin{oldenumerate}[topsep=0pt]}{\end{oldenumerate}}

\date{}

\makeatother

  \addto\captionsamerican{\renewcommand{\corollaryname}{Corollary}}
  \addto\captionsamerican{\renewcommand{\definitionname}{Definition}}
  \addto\captionsamerican{\renewcommand{\examplename}{Example}}
  \addto\captionsamerican{\renewcommand{\factname}{Fact}}
  \addto\captionsamerican{\renewcommand{\lemmaname}{Lemma}}
  \addto\captionsamerican{\renewcommand{\propositionname}{Proposition}}
  \addto\captionsamerican{\renewcommand{\remarkname}{Remark}}
  \addto\captionsamerican{\renewcommand{\theoremname}{Theorem}}
  \addto\captionsenglish{\renewcommand{\corollaryname}{Corollary}}
  \addto\captionsenglish{\renewcommand{\definitionname}{Definition}}
  \addto\captionsenglish{\renewcommand{\examplename}{Example}}
  \addto\captionsenglish{\renewcommand{\factname}{Fact}}
  \addto\captionsenglish{\renewcommand{\lemmaname}{Lemma}}
  \addto\captionsenglish{\renewcommand{\propositionname}{Proposition}}
  \addto\captionsenglish{\renewcommand{\remarkname}{Remark}}
  \addto\captionsenglish{\renewcommand{\theoremname}{Theorem}}
  \providecommand{\corollaryname}{Corollary}
  \providecommand{\definitionname}{Definition}
  \providecommand{\examplename}{Example}
  \providecommand{\factname}{Fact}
  \providecommand{\lemmaname}{Lemma}
  \providecommand{\propositionname}{Proposition}
  \providecommand{\remarkname}{Remark}
\providecommand{\theoremname}{Theorem}

\begin{document}
\selectlanguage{american}%
\global\long\def\epsilon{\varepsilon}

\global\long\def\phi{\varphi}

\global\long\def\RR{\mathbb{R}}

\global\long\def\ind{\clubsuit}

\global\long\def\II{\mathcal{I}}

\global\long\def\rec{\mathcal{R}^{n}}

\global\long\def\kon{\mathcal{K}_{0}^{n}}

\global\long\def\FF{\mathcal{F}^{n}}

\global\long\def\inte{\operatorname{int}}

\global\long\def\ext{\operatorname{Ext}}

\global\long\def\inn{\operatorname{Inn_{S}}}

\global\long\def\conv{\operatorname{conv}}

\global\long\def\proj{\operatorname{Proj}}

\selectlanguage{english}%
\global\long\def\dd{\mathrm{d}}

\selectlanguage{american}%

\title{Reciprocals and Flowers in Convexity}

\author{Emanuel Milman, Vitali Milman, Liran Rotem}
\maketitle
\selectlanguage{english}%
\begin{abstract}
We study new classes of convex bodies and star bodies with unusual
properties. First we define the class of reciprocal bodies, which
may be viewed as convex bodies of the form ``$1/K$''. The map $K\mapsto K^{\prime}$
sending a body to its reciprocal is a duality on the class of reciprocal
bodies, and we study its properties. 

To connect this new map with the classic polarity we use another construction,
associating to each convex body $K$ a star body which we call its
flower and denote by $K^{\ind}$. The mapping $K\mapsto K^{\ind}$
is a bijection between the class $\kon$ of convex bodies and the
class $\FF$ of flowers. Even though flowers are in general not convex,
their study is very useful to the study of convex geometry. For example,
we show that the polarity map $\circ:\kon\to\kon$ decomposes into
two separate bijections: First our flower map $\ind:\kon\to\FF$,
followed by a slight modification $\Phi$ of the spherical inversion
which maps $\FF$ back to $\kon$. Each of these maps has its own
properties, which combine to create the various properties of the
polarity map.

We study the various relations between the four maps $\prime$, $\circ$,
$\ind$ and $\Phi$ and use these relations to derive some of their
properties. For example, we show that a convex body $K$ is a reciprocal
body if and only if its flower $K^{\ind}$ is convex. 

We show that the class $\FF$ has a very rich structure, and is closed
under many operations, including the Minkowski addition. This structure
has corollaries for the other maps which we study. For example, we
show that if $K$ and $T$ are reciprocal bodies so is their ``harmonic
sum'' $(K^{\circ}+T^{\circ})^{\circ}$. We also show that the volume
$\left|\left(\sum_{i}\lambda_{i}K_{i}\right)^{\ind}\right|$ is a
homogeneous polynomial in the $\lambda_{i}$'s, whose coefficients
can be called ``$\ind$-type mixed volumes''. These mixed volumes
satisfy natural geometric inequalities, such as an elliptic Alexandrov-Fenchel
inequality. More geometric inequalities are also derived.
\end{abstract}
\selectlanguage{american}%

\section{Introduction}

\selectlanguage{english}%
In this paper we study new classes of convex bodies and star bodies
in $\RR^{n}$ with some unusual properties. We will provide precise
definitions below, but let us first describe the general program of
what will follow. 

One of our new classes, ``reciprocal'' bodies, may be viewed as
bodies of the form ``$\frac{1}{K}$'' for a convex body $K$. They
appear as the image of a new ``quasi-duality'' operation on the
class $\kon$ of convex bodies. We denote this new map by $K\mapsto K^{\prime}$.
This operation reverses order (with respect to inclusions) and has
the property $K'''=K'$. Hence the map $^{\prime}$ is indeed a duality
on its image.

This new operation is connected to the classical operation of polarity
$\circ:K\mapsto K^{\circ}$ via another construction, which we call
simply the ``flower'' of a body $K$ and denote by $\ind:K\mapsto K^{\ind}$.
We provide the definition of $K^{\ind}$ in Definition \ref{def:indicatrix}
below, but an equivalent description which sheds light on the \textquotedbl{}flower\textquotedbl{}
nomenclature is
\[
K^{\ind}=\bigcup\left\{ B\left(\frac{x}{2},\frac{\left|x\right|}{2}\right):\ x\in K\right\} 
\]
 (see Proposition \ref{prop:repr-formulas}). Here $B(y,r)$ is the
Euclidean ball with center $y\in\RR^{n}$ and radius $r\ge0$. In
other words, $K^{\ind}$ is the union of all balls passing through
the origin having diameter $[0,x]$ with $x\in K$. 

In general, $K^{\ind}$ is a star body which is not necessarily convex.
The flower of a convex body was previously studied for very different
reasons in the field of stochastic geometry – see Remark \ref{rem:voronoi}.
We show that our new map $^{\prime}$ is precisely $K^{\prime}=\left(K^{\ind}\right)^{\circ}$.
We also show that $K$ belongs to the image of $^{\prime}$, i.e.
$K$ is a reciprocal body, if and only if $K^{\ind}$ is convex. This
means that such reciprocal bodies are in some sense ``more convex''
than other convex bodies, and can also be thought of as ``doubly
convex'' bodies.

Interestingly, the flower map $\ind$ is also connected to the $n$-dimensional
spherical inversion $\Phi$ when applied to star bodies ($\Phi$ is
defined by applying the pointwise map $\II(x)=\frac{x}{\left|x\right|^{2}}$
and taking set complement – see Definition \ref{def:spherical-duality}).
We describe the class of convex bodies on which $\Phi$ preserves
convexity. 

The method of study of these questions looks novel and some of the
results are not intuitive. Just as an example, we show that if $\Phi(A)$
and $\Phi(B)$ are convex (for some star bodies $A$ and $B$) then
$\Phi(A+B)$ is convex as well, where $A+B$ is the Minkowski addition
(see Corollary \ref{cor:inv-convex-sum}).

The family $\FF$ of flowers should play a central role in the study
of convexity. It has a very rich structure. For example, it is closed
under the Minkowski addition, and is also preserved by orthogonal
projections and sections. ``Flower mixed volumes'' also exist and,
perhaps most interestingly, we have a decomposition of the classical
polarity operation as 
\[
\kon\stackrel{\ind}{\longrightarrow}\FF\stackrel{\Phi}{\longrightarrow}\kon.
\]
Here the maps $\ind$ and $\Phi$ are $1$-1 and onto, and we have
$\circ=\Phi\ind$ in the sense that $K^{\circ}=\Phi\left(K^{\ind}\right)$
for all $K\in\kon$. 

The class of reciprocal bodies also looks interesting. No polytope
belongs to this class, and no centrally symmetric ellipsoids (besides
Euclidean balls centered at $0$). At the same time this class is
clearly important, as seen from its properties and the fact that it
coincides with the ``doubly convex'' bodies. We provide several
$2$-dimensional pictures to help create some intuition about this
class of reciprocal bodies and about the class of flowers.

\selectlanguage{american}%
To make the above claims more precise, let us now give some basic
definitions and fix our notation. The reader may consult \cite{Schneider2013}
for more information. By a \emph{convex body} in $\RR^{n}$ we mean
a set $K\subseteq\RR^{n}$ which is closed and convex. We will always
assume further that $0\in K$, but we do not assume that $K$ is compact
or has non-empty interior. We denote the set of all such bodies by
$\kon$. The \emph{support function} of $K$ is the function $h_{K}:S^{n-1}\to[0,\infty]$
defined by $h_{K}(\theta)=\sup_{x\in K}\left\langle x,\theta\right\rangle $.
Here $S^{n-1}=\left\{ \theta\in\RR^{n}:\ \left|\theta\right|=1\right\} $
is the unit Euclidean sphere, and $\left\langle \cdot,\cdot\right\rangle $
is the standard scalar product on $\RR^{n}$. The function $h_{K}$
uniquely defines the body $K$. 

The Minkowski sum of two convex bodies is defined by 
\[
K+T=\overline{\left\{ x+y:\ x\in K,\ y\in T\right\} }
\]
 (the closure is not needed if $K$ or $T$ is compact). The homothety
operation is defined by $\lambda K=\left\{ \lambda x:\ x\in K\right\} $.
These operations are related to the support function by the identity
$h_{\lambda K+T}=\lambda h_{K}+h_{T}$. 

We say that $A\subseteq\RR^{n}$ is a star set if $A$ is non-empty
and $x\in A$ implies that $\lambda x\in A$ for all $0\le\lambda\le1$.
The \emph{radial function} $r_{A}:S^{n-1}\to[0,\infty]$ of $A$ is
defined by $r_{A}(\theta)=\sup\left\{ \lambda\ge0:\ \lambda\theta\in A\right\} $.
For us, a \emph{star body} is simply a star set which is radially
closed, in the sense that $r_{A}(\theta)\theta\in A$ for all directions
$\theta\in S^{n-1}$ satisfying $r_{A}(\theta)<\infty$. For such
bodies $r_{A}$ uniquely defines $A$. 

The \emph{polarity map} $\circ:\kon\to\kon$ maps every body $K$
to its polar 
\begin{equation}
K^{\circ}=\left\{ y\in\RR^{n}:\ \left\langle x,y\right\rangle \le1\text{ for all }x\in K\right\} .\label{eq:polar-def}
\end{equation}
 It follows that $h_{K}=\frac{1}{r_{K^{\circ}}}$. The polarity map
is a duality in the following sense: 
\begin{itemize}
\item It is order reversing: If $K\subseteq T$ then $K^{\circ}\supseteq T^{\circ}$.
\item It is an involution: $K^{\circ\circ}=K$ for all $K\in\kon$ (if $A$
is only a star body, then $A^{\circ\circ}$ is the closed convex hull
of $A$). 
\end{itemize}
In fact, it was proved in \cite{Artstein-Avidan2008} that the polarity
map is essentially the \emph{only} duality on $\kon$. Similar results
on different classes of convex bodies were proved earlier in \cite{Gruber1992}
and \cite{Boroczky2008}. 

The structure of a set equipped with a duality relation is common
in mathematics. A basic example is the set $[0,\infty]$ equipped
with the inversion $x\mapsto x^{-1}$ (we set of course $0^{-1}=\infty$
and $\infty^{-1}=0$). Following this analogy, one may think of $K^{\circ}$
as a certain inverse ``$K^{-1}$''. This point of view can indeed
be useful – see for example \cite{Molchanov2014} and \cite{Milman2016a}.

However, in recent works (\cite{Milman2017a}, \cite{Milman2018}),
the authors discussed the application of functions such as $x\mapsto x^{\alpha}$
($0\le\alpha\le1$) and $x\mapsto\log x$ to convex bodies. Applying
the same idea to the inversion $x\mapsto\frac{1}{x}$, we obtain a
new notion of the reciprocal body ``$K^{-1}$'' . Recall that given
a function $g:S^{n-1}\to[0,\infty]$, its \emph{Alexandrov body},
or \emph{Wulff shape}, is defined by 
\[
A\left[g\right]=\left\{ x\in\RR^{n}:\ \left\langle x,\theta\right\rangle \le g(\theta)\text{ for all }\theta\in S^{n-1}\right\} .
\]
 In other words, $A\left[g\right]$ is the biggest convex body such
that $h_{A[g]}\le g$. In particular, for every convex body $K$ we
have $K=A\left[h_{K}\right]$. We may now define:
\selectlanguage{english}%
\begin{defn}
Given $K\in\kon$, the \emph{reciprocal body} $K'\in\kon$ is defined
by $K'=A\left[\frac{1}{h_{K}}\right].$

More explicitly, we have 

\[
K'=\bigcap_{\theta\in S^{n-1}}H^{-}\left(\theta,h_{K}(\theta)^{-1}\right),
\]
where $H^{-}(\theta,c)=\left\{ x\in\RR^{n}:\ \left\langle x,\theta\right\rangle \le c\right\} .$ 
\end{defn}
The idea of constructing new interesting convex bodies as Alexandrov
bodies is not new. As one important recent example, Böröczky, Lutwak,
Yang and Zhang consider in \cite{Boroczky2012} the body $A\left[h_{K}^{1-\lambda}h_{L}^{\lambda}\right]$,
which they call the $\lambda$-logarithmic mean of $K$ and $L$. 

Figure \ref{fig:reciprocals} depicts some simple convex bodies in
$\RR^{2}$ and their reciprocal. Some basic properties of the reciprocal
map $K\mapsto K^{\prime}$ are immediate from the definition:

\begin{figure}
\includegraphics[scale=0.42]{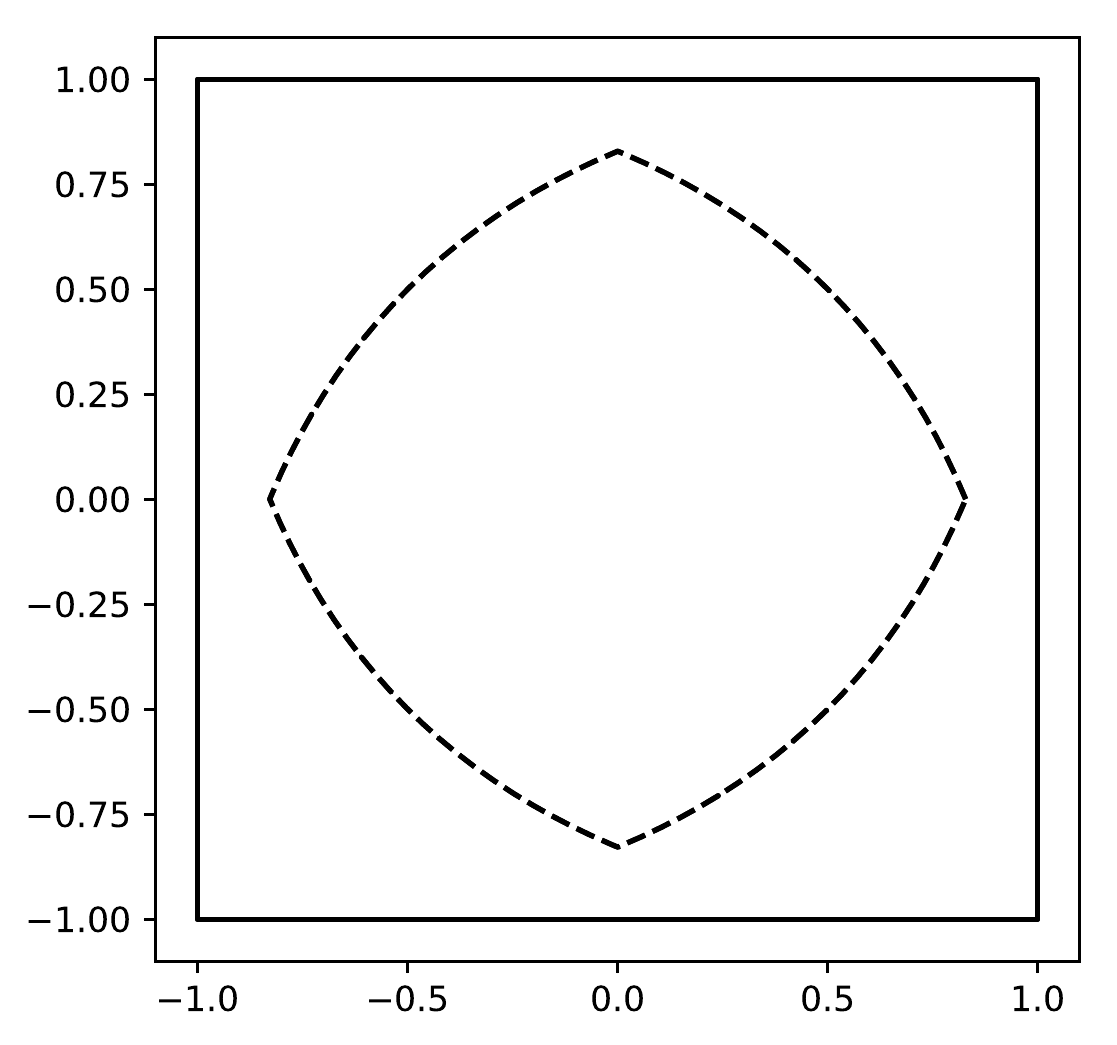}\hfill{}\includegraphics[scale=0.42]{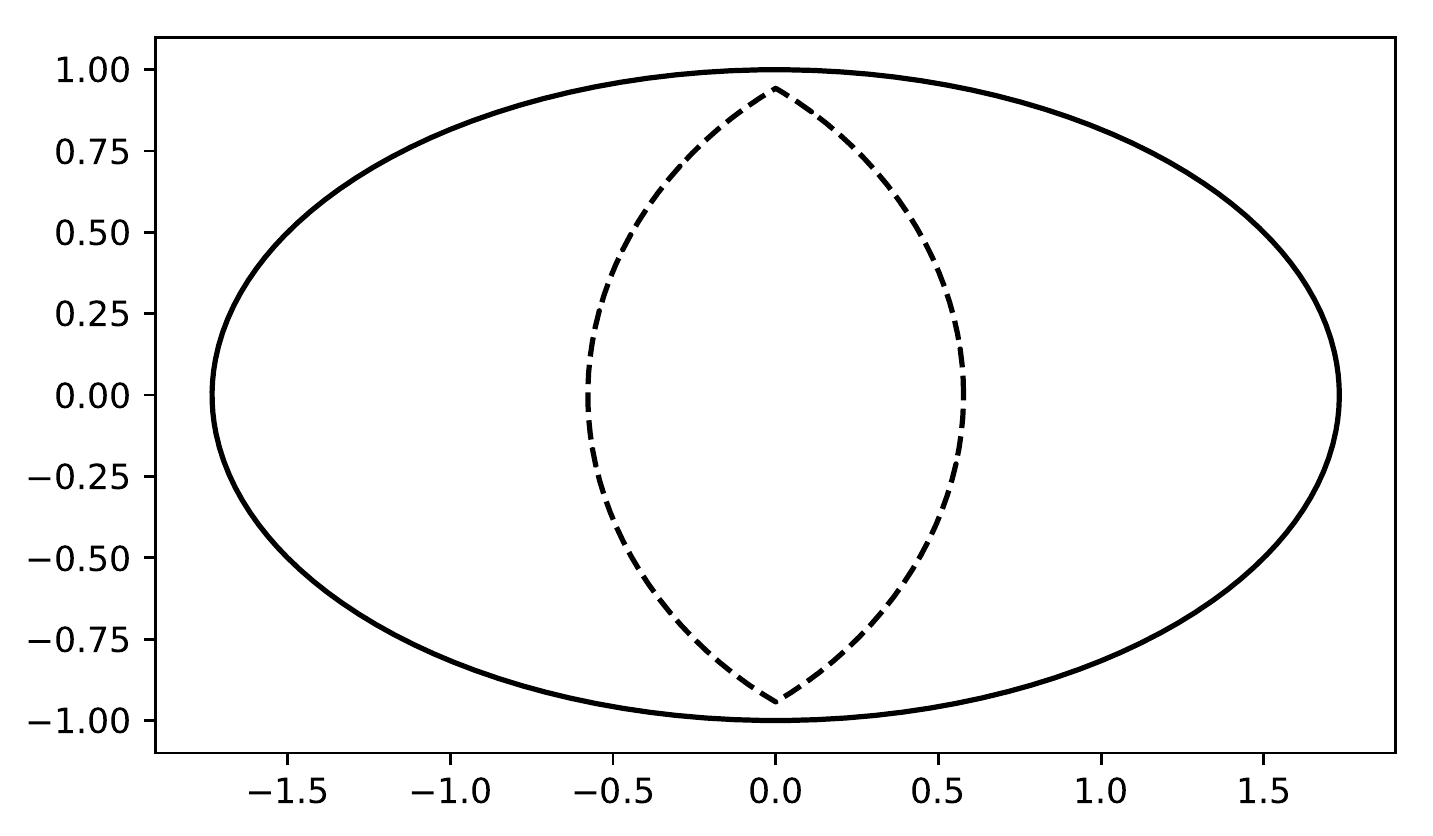}\hfill{}\includegraphics[scale=0.42]{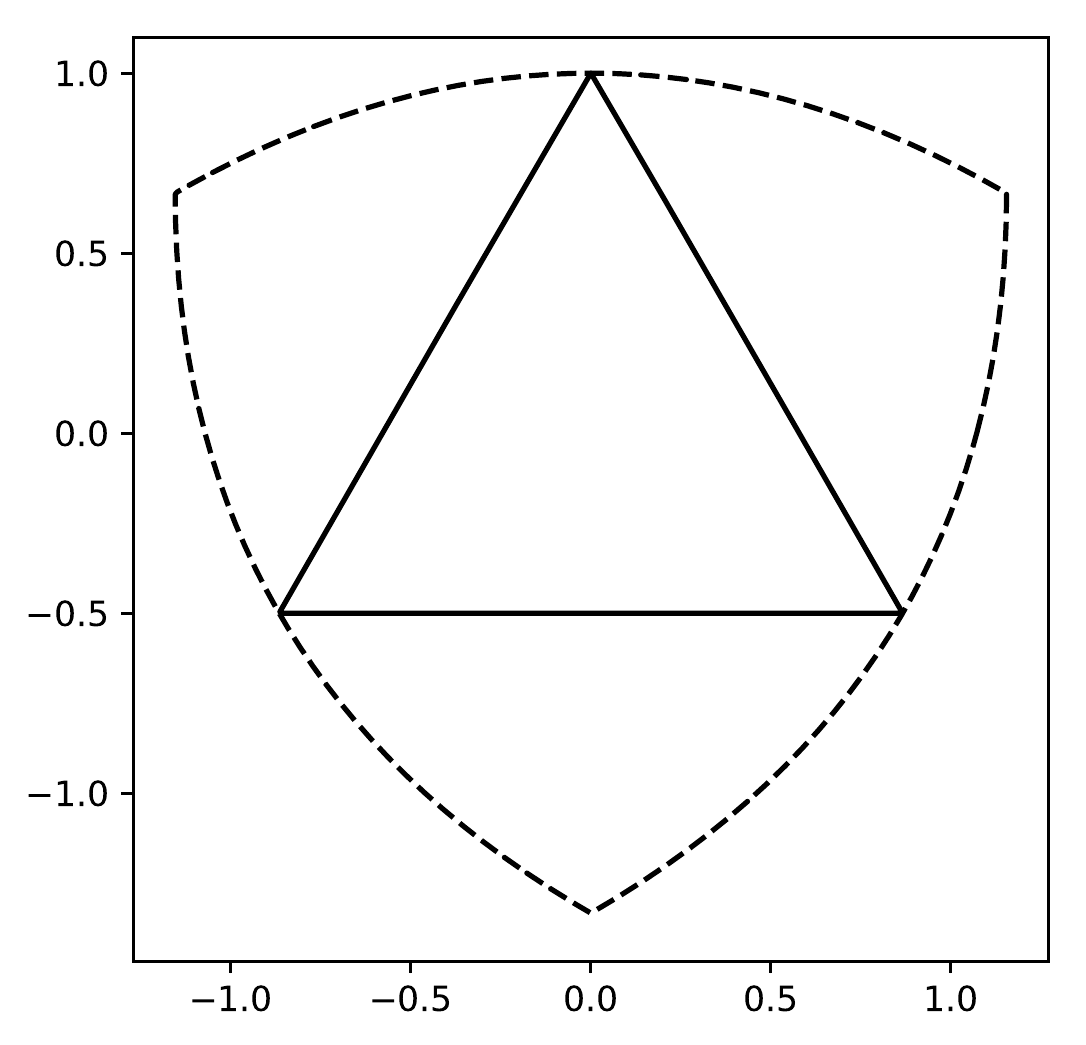}

\caption{\label{fig:reciprocals}convex bodies (solid) and their reciprocals
(dashed)}
\end{figure}

\begin{prop}
\label{prop:basic-prop}For all $K,T\in\kon$ we have:
\begin{enumerate}
\item \label{enu:basic-prime-polar}$K'\subseteq K^{\circ}$, with an equality
if and only if $K$ is a Euclidean ball. 
\item \label{enu:basic-order}If $K\supseteq T$ then $K'\subseteq T'$.
\item \label{enu:basic-double}$K''\supseteq K$.
\item \label{enu:basic-triple}$K'''=K'$.
\end{enumerate}
\end{prop}
\begin{proof}
For \eqref{enu:basic-prime-polar}, note that for every $\theta\in S^{n-1}$
we have $1=\left\langle \theta,\theta\right\rangle \le h_{K}(\theta)h_{K^{\circ}}(\theta)$.
Hence $K^{\circ}=A\left[h_{K^{\circ}}\right]\ge A\left[\frac{1}{h_{K}}\right]=K'$.
An equality $K'=K^{\circ}$ implies that $h_{K^{\circ}}=\frac{1}{h_{K}}$,
or equivalently $r_{K}=\frac{1}{h_{K^{\circ}}}=h_{K}$. This implies
that $K$ is a ball.

Property \eqref{enu:basic-order} is obvious from the definition.

For property \eqref{enu:basic-double}, we know that $h_{K'}\le\frac{1}{h_{K}}$
so $K''=A\left[\frac{1}{h_{K'}}\right]\ge A\left[h_{K}\right]=K$. 

Finally, \eqref{enu:basic-triple} is a formal consequence of \eqref{enu:basic-order}
and \eqref{enu:basic-double}: We know that $K''\supseteq K$, so
$K'''\subseteq K'$. On the other hand applying \eqref{enu:basic-double}
to $K'$ gives $K'''\supseteq K'$.
\end{proof}
Let us write 

\[
{\cal R}^{n}=\left\{ K^{\prime}:\ K\in\kon\right\} .
\]
Note that properties \eqref{enu:basic-order} and \eqref{enu:basic-triple}
above imply that $^{\prime}$ is a duality on the class $\rec$. Also
note that $K\in\rec$ if and only if $K''=K$. 

Our next goal is to give an alternative description of the reciprocal
body $K^{\prime}$. Towards this goal we define:
\begin{defn}
\label{def:indicatrix}
\begin{enumerate}
\item For a convex body $K\in\kon$ we denote by $K^{\ind}$ the star body
with radial function $r_{K^{\ind}}=h_{K}$. 
\item We say that a star body $A\subseteq\RR^{n}$ is a \emph{flower} if
$A=\bigcup_{x\in C}B\left(\frac{x}{2},\frac{\left|x\right|}{2}\right)$,
where $C\subseteq\RR^{n}$ is some closed set. The class of all flowers
in $\RR^{n}$ is denoted by $\FF$.
\end{enumerate}
\end{defn}
The two parts of the definition are related by the following:
\begin{thm}
For every $K\in\kon$ we have $K^{\ind}\in\FF$. Moreover, the map
$\ind:\kon\to\FF$ is one to one and onto. Equivalently, every flower
$A$ is of the form $A=K^{\ind}$ for a unique $K\in\kon$; We have
$A=\bigcup_{x\in K}B\left(\frac{x}{2},\frac{\left|x\right|}{2}\right)$,
and we simply say that $A$ is the flower of $K$. 
\end{thm}
This theorem is a combination of Proposition \ref{prop:indicatrix-prop}\eqref{enu:ind-unique},
Proposition \ref{prop:repr-formulas}, and Remark \ref{rem:non-convex}. 

As we will see flowers play an important role in connecting the reciprocity
map to the polarity map. Note that in general $K^{\ind}$ is not convex.
Figure \ref{fig:indicatrices} depicts the flowers of some convex
bodies in $\RR^{2}$. Another example that will be important in the
sequel is the following:

\begin{figure}
\includegraphics[scale=0.42]{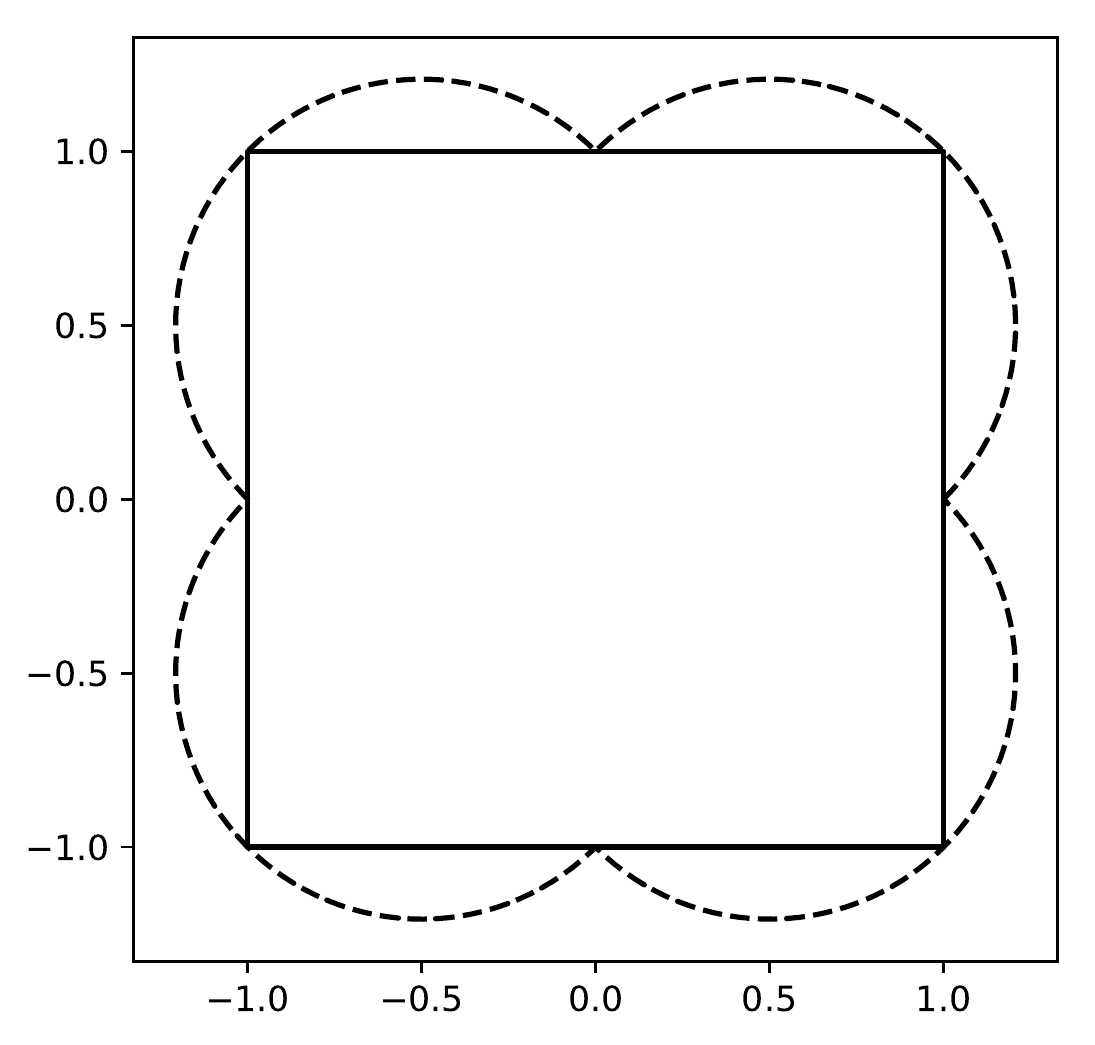}\hfill{}\includegraphics[scale=0.42]{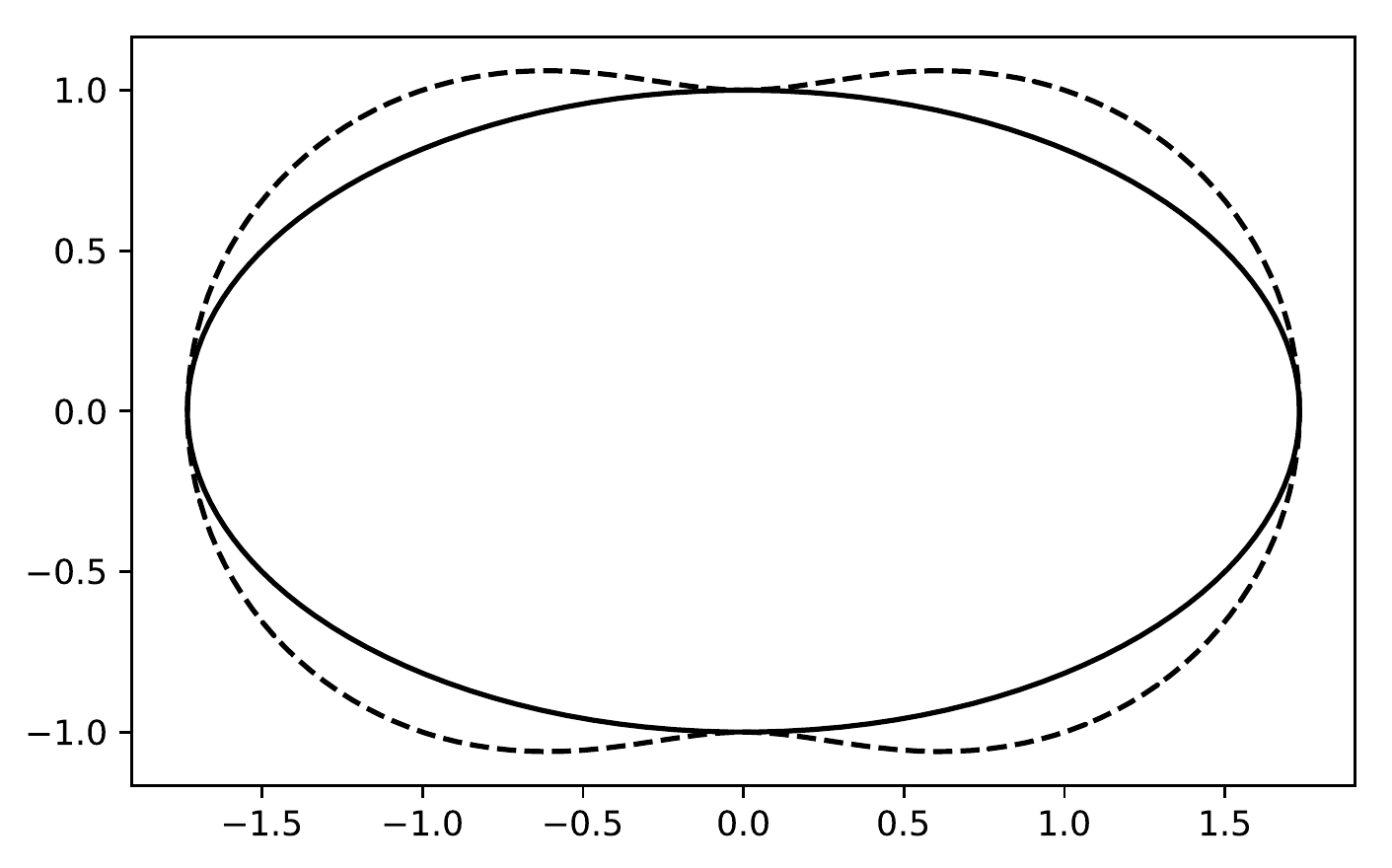}\hfill{}\includegraphics[scale=0.42]{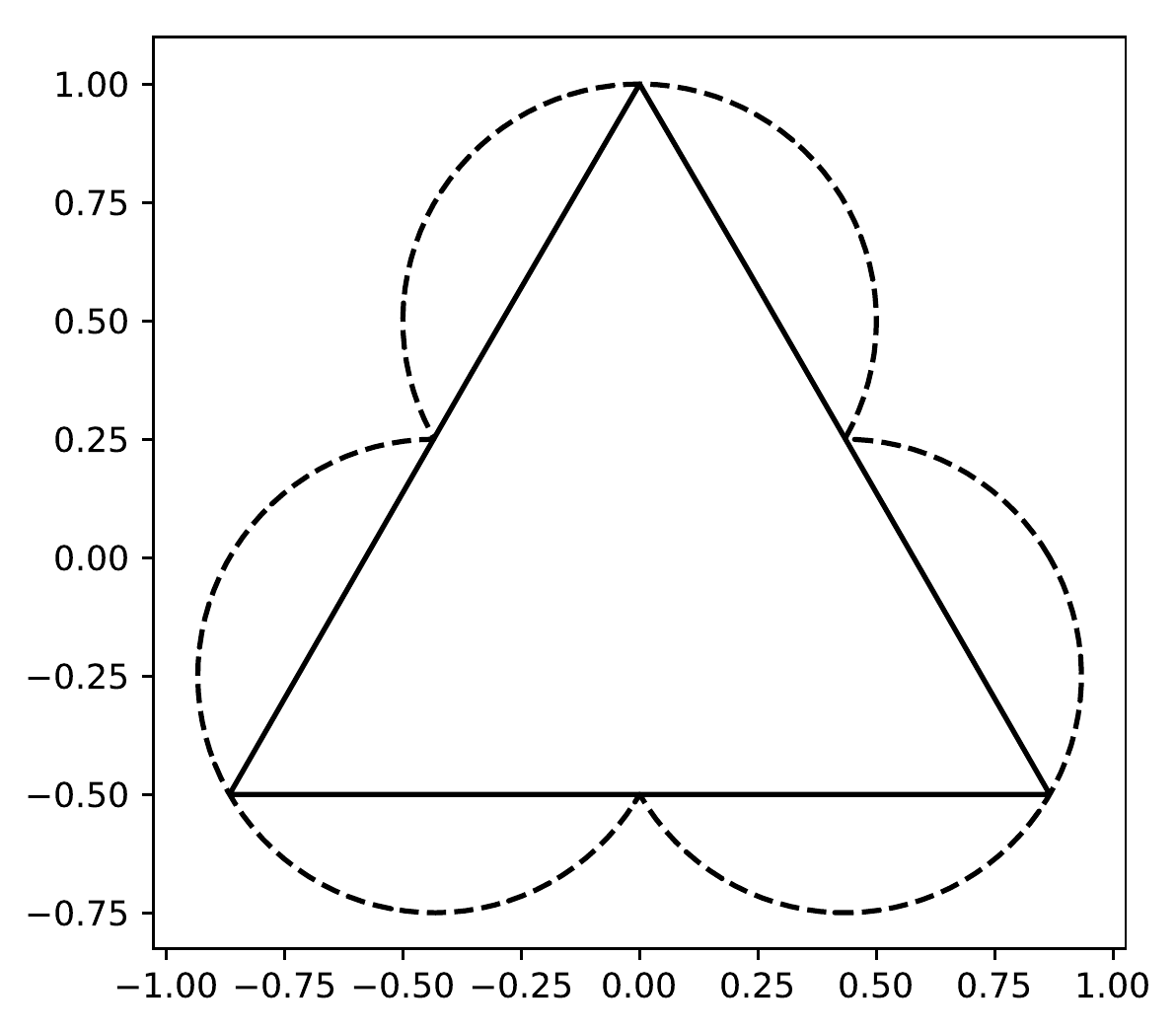}

\caption{\label{fig:indicatrices}convex bodies (solid) and their flowers (dashed)}
\end{figure}

\begin{example}
\label{exa:indicatrix-interval}For $x\in\RR^{n}$ write $[0,x]=\left\{ \lambda x:\ 0\le\lambda\le1\right\} $.
Also denote the Euclidean ball with center $x$ and radius $r>0$
by $B(x,r)$, and write $B_{x}=B\left(\frac{x}{2},\frac{\left|x\right|}{2}\right)$.
Then $[0,x]^{\ind}=B_{x}$. Indeed, a direct computation gives 
\[
h_{[0,x]}(\theta)=r_{B_{x}}(\theta)=\max\left\{ \left\langle x,\theta\right\rangle ,0\right\} .
\]
The identity $[0,x]^{\ind}=B_{x}$ is also a classical theorem in
geometry sometimes referred to as Thales's theorem: If an interval
$[a,b]\subseteq\RR^{n}$ is a diameter of a ball $B$, then $\partial B$
is precisely the set of points $y$ such that $\measuredangle ayb=90^{\circ}$. 
\end{example}
The polarity map, the reciprocal map and the flower are all related
via the following formula:
\begin{prop}
\label{prop:prime-as-polar}For every $K\in\kon$ we have $\left(K^{\ind}\right)^{\circ}=K^{\prime}$.
\end{prop}
Note that even though in general $K^{\ind}\notin\kon$, we may still
compute its polar using \eqref{eq:polar-def}.
\begin{proof}
By definition $x\in\left(K^{\ind}\right)^{\circ}$ if and only if
$\left\langle x,y\right\rangle \le1$ for all $y\in K^{\ind}$. It
is obviously enough to check this for $y\in\partial K^{\ind}$, i.e.
$y=r_{K^{\ind}}(\theta)\theta=h_{K}(\theta)\theta$ for some $\theta\in S^{n-1}$. 

Hence $x\in\left(K^{\ind}\right)^{\circ}$ if and only if for all
$\theta\in S^{n-1}$ we have $\left\langle x,h_{K}(\theta)\theta\right\rangle \le1$,
or $\left\langle x,\theta\right\rangle \le\frac{1}{h_{K}(\theta)}$.
This means that $x\in A\left[\frac{1}{h_{K}}\right]=K^{\prime}$.
\end{proof}
\begin{rem}
\label{rem:voronoi}The flower of a convex body was studied in stochastic
geometry under the name ``Voronoi Flower'' (see e.g. \cite{Spodarev2013}).
The reason for the name is the following relation to Voronoi tessellations:
For a discrete set of points $P\subseteq\RR^{n}$, consider the (open)
Voronoi cell
\[
Z=\left\{ x\in\RR^{n}:\ \left|x-0\right|<\left|x-y\right|\text{ for all }y\in P\right\} .
\]
Then for any convex body $K$ we have $Z\supseteq K$ if and only
if $P\cap\left(2K^{\ind}\right)=\emptyset$. It follows that if for
example $P$ is chosen according to a homogeneous Poisson point process,
then the probability that $Z\supseteq K$ is computable from the volume
of $K^{\ind}$. 
\end{rem}
In Section \ref{sec:properties} we discuss basic properties of the
flower map $\ind$ and prove representation formulas for both $K^{\ind}$
and $K^{\prime}$. We also study the pre-images of a body $K\in\rec$
under the reciprocity map. Since $\prime$ is not a duality on all
of $\kon$, the set of pre-images 
\[
\left\{ A\in\kon:\ A'=K\right\} 
\]
may in general contain more than one body. We study this set, and
prove the following results:
\begin{thm}
\begin{enumerate}
\item If $K\in\rec$ is a \emph{smooth} convex body then $K=A'$ for a \emph{unique}
$A\in\kon$. 
\item For a general $K\in\rec$, the set $\left\{ A\in\kon:\ A'=K\right\} $
is a convex subset of $\kon$. 
\end{enumerate}
\end{thm}
The main goal of Section \ref{sec:main-proof} is to prove the following
theorem, characterizing the class $\rec$ of reciprocal bodies:
\begin{thm}
\label{thm:double-convex}$K\in\rec$ if and only if $K^{\ind}$ is
convex.
\end{thm}
As a corollary we obtain:
\begin{cor}
For every $K\in\rec$ and every subspace $E\subseteq\RR^{n}$ one
has $\left(\proj_{E}K\right)^{\prime}=\proj_{E}K^{\prime}$, where
$\proj_{E}$ denotes the orthogonal projection onto $E$. 
\end{cor}
We will prove Theorem \ref{thm:double-convex} by connecting the various
maps we constructed so far with another duality on the class of star-bodies:
\begin{defn}
\label{def:spherical-duality}
\begin{enumerate}
\item Let $\mathcal{\II}:\RR^{n}\setminus\left\{ 0\right\} \to\RR^{n}\setminus\left\{ 0\right\} $
denote the spherical inversion $\mathcal{\II}(x)=\frac{x}{\left|x\right|^{2}}$.
\item Given a star body $A$, we denote by $\Phi(A)$ the star body with
radial function $r_{\Phi(A)}=\frac{1}{r_{A}}$. 
\end{enumerate}
\end{defn}
The map $A\mapsto\Phi(A)$ is obviously a duality on the class of
star bodies. It is sometimes called star duality and denoted by $A^{\ast}$
(see \cite{Moszynska1999}), but we will prefer the notation $\Phi(A)$.
Note that $\Phi$ is ``essentially the same'' as the pointwise map
$\II$ in the sense that $\partial\Phi(A)=\II\left(\partial A\right)$,
but $\II$ maps the interior of $A$ to the exterior of $\Phi(A)$
and vice versa. Here by the boundary $\partial A$ of a star body
$A$ we mean
\[
\partial A=\left\{ r_{A}(\theta)\theta:\ \theta\in S^{n-1}\text{ such that }0<r_{A}(\theta)<\infty\right\} .
\]
One interesting relation between $\Phi$ and our previous definitions
is the following (see Propositions \ref{prop:op-identities}\eqref{enu:op-phi-star}
and \ref{prop:phi-convex}):
\begin{thm}
\label{thm:polar-decomposition}$\Phi$ is a bijection between $\kon$
and $\FF$. Moreover, the polarity map decomposes as 
\[
\circ:\kon\stackrel{\ind}{\longrightarrow}\FF\stackrel{\Phi}{\longrightarrow}\kon,
\]
 in the sense that $\Phi\left(K^{\ind}\right)=K^{\circ}$ for all
$K\in\kon$.
\end{thm}
In Section \ref{sec:phi-properties} we use the results of Section
\ref{sec:main-proof} to further study the class of flowers, with
applications to the study of reciprocity and the map $\Phi$. First
we understand when the map $\Phi$ preserves convexity. By Theorem
\ref{thm:polar-decomposition}, as $\Phi$ is an involution, we know
that $\Phi(A)$ is convex if and only if $A$ is a flower. When $A$
is in addition convex, we have:
\begin{thm}
If $K\in\kon$ then $\Phi(K)$ is convex if and only if $K^{\circ}\in\rec$.
\end{thm}
(See Proposition \ref{prop:phi-convex}). We then show that the class
$\FF$ has a lot of structure:
\begin{thm}
Fix $A,B\in\FF$ and a linear subspace $E\subseteq\RR^{n}$. Then
$A+B$ and $\conv A$ are flowers in $\RR^{n}$, and $A\cap E$ and
$\proj_{E}A$ are flowers in $E$. 
\end{thm}
(See Propositions \ref{thm:flower-sum}, \ref{prop:flower-proj} and
\ref{prop:flower-conv}). As corollaries we obtain:
\begin{cor}
\begin{enumerate}
\item If $K,T\in\rec$ then $\left(K^{\circ}+T^{\circ}\right)^{\circ}\in\rec$.
\item If $K,T$ are convex bodies then $\Phi\left(\Phi(K)+\Phi(T)\right)$
is also convex.
\end{enumerate}
\end{cor}
As another corollary we construct a new addition $\oplus$ on $\kon$
such that the class $\rec$ is closed under $\oplus$. Moreover, when
restricted to $\rec$, this new addition has all properties one may
expect: it is associative, commutative and monotone, it has $\left\{ 0\right\} $
as an identity element, and it satisfies $\lambda K\oplus\mu K=\left(\lambda+\mu\right)K$. 

The final Section \ref{sec:geometric-Ineq} is devoted to the study
of inequalities. We begin by showing that the maps $\ind,\Phi$ and
$\prime$ are all convex in appropriate senses. We also study the
functional $K\mapsto\left|K^{\ind}\right|$, where $\left|\cdot\right|$
denotes the volume. We prove results that are analogous to Minkowski's
theorem of polynomiality of volume and to the Alexandrov-Fenchel inequality:
\begin{thm}
Fix $K_{1},K_{2},\ldots,K_{m}\in\kon$. Then 
\[
\left|\left(\lambda_{1}K_{1}+\lambda_{2}K_{2}+\cdots+\lambda_{m}K_{m}\right)^{\ind}\right|=\sum_{i_{1},i_{2},\ldots,i_{n}=1}^{m}V^{\ind}(K_{i_{1}},K_{i_{2}},\ldots,K_{i_{n}})\cdot\lambda_{i_{1}}\lambda_{i_{2}}\cdots\lambda_{i_{n}},
\]
where the coefficients are given by
\[
V^{\ind}(K_{1},K_{2},\ldots,K_{n})=\left|B_{2}^{n}\right|\cdot\int_{S^{n-1}}h_{K_{1}}(\theta)h_{K_{2}}(\theta)\cdots h_{K_{n}}(\theta)\dd\sigma(\theta)
\]
 (Here $B_{2}^{n}$ denotes the unit Euclidean ball). Moreover, for
every $K_{1},K_{2},\ldots,K_{n}\in\kon$ we have 
\[
V^{\ind}(K_{1},K_{2},K_{3},\ldots,K_{n})^{2}\le V^{\ind}\left(K_{1},K_{1},K_{3},\ldots K_{n}\right)\cdot V^{\ind}\left(K_{2},K_{2},K_{3},\ldots,K_{n}\right).
\]
\end{thm}
These results and their proofs are similar in spirit to the dual Brunn–Minkowski
theory which was developed by Lutwak in \cite{Lutwak1975}. We also
prove a Kubota type formula for the new $\ind$-quermassintegrals,
and use it to compare them with the classical definition.

\paragraph*{Acknowledgments: }

The authors would like to thank M. Gromov and R. Schneider for a useful
exchange of messages regarding this paper. They would also like to
thank R. Gardner and D. Hug for introducing them to the useful references. 

\section{\label{sec:properties}Properties of reciprocity and flowers}

We begin this section with some basic properties of flowers:
\begin{prop}
\label{prop:indicatrix-prop}
\begin{enumerate}
\item \label{enu:ind-inclusion}For every $K\in\kon$ we have $K^{\ind}\supseteq K$,
with equality if and only if $K$ is an Euclidean ball.
\item \label{enu:ind-unique}If $K^{\ind}=T^{\ind}$ for $K,T\in\kon$ then
$K=T$. 
\item \label{enu:ind-union}Let $\left\{ K_{i}\right\} _{i\in I}$ be a
family of convex bodies. Then $\left(\conv\left(\bigcup_{i\in I}K_{i}\right)\right)^{\ind}=\bigcup_{i\in I}K_{i}^{\ind}$. 
\item \label{enu:ind-proj}For every $K\in\kon$ and every subspace $E\subseteq\RR^{n}$
we have $\left(\proj_{E}K\right)^{\ind}=K^{\ind}\cap E$ (where the
$\ind$ on the left hand side is taken inside the subspace $E$). 
\end{enumerate}
\end{prop}
\begin{proof}
For \eqref{enu:ind-inclusion} we have $r_{K^{\ind}}=h_{K}\ge r_{K}$.
The equality case is the same as in Proposition \ref{prop:basic-prop}(\ref{enu:basic-prime-polar}). 

\eqref{enu:ind-unique} is obvious since $h_{K}$ uniquely defines
$K$. For \eqref{enu:ind-union}, write $A=\conv\left(\bigcup_{i\in I}K_{i}\right)$
and $B=\bigcup_{i\in I}K_{i}^{\ind}$. Then
\[
r_{A^{\ind}}=h_{A}=\max_{i\in I}h_{K_{i}}=\max_{i\in I}r_{K_{i}^{\ind}}=r_{B},
\]
 so $A^{\ind}=B$. 

Finally, for \eqref{enu:ind-proj}, since both bodies are inside $E$
its enough to check that their radial functions coincide in $E$.
But if $\theta\in S^{n-1}\cap E$ then 
\[
r_{\left(\proj_{E}K\right)^{\ind}}(\theta)=h_{\proj_{E}K}(\theta)=h_{K}(\theta)=r_{K^{\ind}}(\theta)=r_{K^{\ind}\cap E}(\theta),
\]
proving the claim.
\end{proof}
We will also need the following computation:
\begin{lem}
Let $B_{x}=B\left(\frac{x}{2},\frac{\left|x\right|}{2}\right)$ be
the ball with center $\frac{x}{2}$ and radius $\frac{\left|x\right|}{2}$.
Let $P_{x}$ be the paraboloid,
\[
P_{x}=\left\{ y\in\RR^{n}:\ \left\langle y,x\right\rangle \le1-\frac{1}{4}\left|x\right|^{2}\left|\text{Proj}_{x^{\perp}}y\right|^{2}\right\} ,
\]
 where $\text{Proj}_{x^{\perp}}$ denotes the orthogonal projection
to the hyperplane orthogonal to $x$. Then $B_{x}^{\circ}=P_{x}$. 
\end{lem}
\begin{proof}
It is enough to prove the result for $x=e_{n}=(0,0,\ldots,0,1)$.
Indeed, we can a write $x=\lambda\cdot u(e_{n})$ for some orthogonal
matrix $u$ and some $\lambda>0$, and then 
\[
\left(B_{x}\right)^{\circ}=\left(\lambda\cdot u\left(B_{e_{n}}\right)\right)^{\circ}=\frac{1}{\lambda}\cdot u\left(B_{e_{n}}^{\circ}\right)=\frac{1}{\lambda}\cdot u\left(P_{e_{n}}\right)=P_{x}.
\]

Write a general point $y\in\RR^{n}$ as $y=(z,t)\in\RR^{n-1}\times\RR$.
Since $B_{x}=[0,x]^{\ind}$ we know that 
\[
r_{B_{e_{n}}}\left(z,t\right)=h_{[0,e_{n}]}(z,t)=\max\left\{ t,0\right\} .
\]
Hence we have 
\begin{align*}
h_{B_{e_{n}}}\left(z,t\right) & =\max_{\theta\in S^{n-1}}\left\langle (z,t),r_{K}(\theta)\theta\right\rangle =\max_{(u,s)\in S^{n-1}}\left\langle (z,t),(u,s)\right\rangle \max\left\{ s,0\right\} \\
 & =\max_{(u,s)\in\RR^{n-1}\times\RR}\left(\frac{\left\langle z,u\right\rangle +ts}{\left|u\right|^{2}+s^{2}}\cdot\max\left\{ s,0\right\} \right).
\end{align*}
It is obviously enough to maximize over $s>0$, and by homogeneity
we may take $s=1$. It is also clear that the maximum is attained
when $u=r\cdot\frac{z}{\left|z\right|}$ for some $r$. Therefore
\[
h_{B_{e_{n}}}(z,t)=\max_{r}\left(\frac{r\left|z\right|+t}{r^{2}+1}\right).
\]
We see that $(z,t)\in B_{e_{n}}^{\circ}$ if and only if for all $r$
we have $\frac{r\left|z\right|+t}{r^{2}+1}\le1$, or $r^{2}-\left|z\right|r+1-t\ge0$.
This happens exactly when $\left|z\right|^{2}-4(1-t)\le0,$ or $t\le1-\frac{\left|z\right|^{2}}{4}$.
Hence $B_{e_{n}}^{\circ}=P_{e_{n}}$ like we wanted.
\end{proof}
Hence we obtain the following descriptions of $K^{\ind}$ and $K^{\prime}$:
\begin{prop}
\label{prop:repr-formulas}For every $K\in\kon$ we have $K^{\ind}=\bigcup_{x\in K}B_{x}$,
and $K^{\prime}=\bigcap_{x\in K}P_{x}$. 
\end{prop}
\begin{proof}
Since $K=\conv\left(\bigcup_{x\in K}[0,x]\right)$, Proposition \ref{prop:indicatrix-prop}\eqref{enu:ind-union}
implies that $K^{\ind}=\bigcup_{x\in K}B_{x}$. Hence 
\[
K'=\left(K^{\ind}\right)^{\circ}=\bigcap_{x\in K}B_{x}^{\circ}=\bigcap_{x\in K}P_{x}.
\]
\end{proof}
\begin{rem}
If $K$ is compact, the same proof shows that it is enough to consider
only $x\in\partial K$. In fact we can do a bit more: recall that
$x\in\partial K$ is an extremal point for $K$ if any representation
$x=(1-\lambda)y+\lambda z$ for $0<\lambda<1$ and $y,z\in K$ implies
that $y=z=x$. Denote the set of extremal points by $\ext(K)$. By
the Krein–Milman theorem\footnote{In the finite dimensional case the Krein–Milman theorem was first
proved by Minkowski. See \cite{Schneider2013} and in particular the
first note of Section 1.4. } we have $K=\conv\left(\bigcup_{x\in\ext(K)}[0,x]\right)$, so $K^{\ind}=\bigcup_{x\in\ext(K)}B_{x}$
and $K^{\prime}=\bigcap_{x\in\ext(K)}P_{x}$. In particular if $K$
is a polytope then $K^{\ind}$ is the union of finitely many balls
and $K'$ is the intersection of finitely many paraboloids.
\end{rem}
\begin{rem}
\label{rem:non-convex}The formulas of Proposition \ref{prop:repr-formulas}
can be used to define $K^{\ind}$ and $K^{\prime}$ for non-convex
sets (say compact). However, it turns out that under such definitions
we have $K^{\ind}=\left(\text{conv}K\right)^{\ind}$ and $K^{\prime}=\left(\conv K\right)^{\prime}$,
so essentially nothing new is gained. To see that $K^{\ind}=\left(\text{conv}K\right)^{\ind}$
note that by the remark above 
\[
\left(\text{conv}K\right)^{\ind}=\bigcup_{x\in\ext(\conv K)}B_{x}\subseteq\bigcup_{x\in K}B_{x}=K^{\ind}.
\]
\end{rem}
Let us now give one application of Proposition \ref{prop:repr-formulas}.
We say that $K\in\kon$ is \emph{smooth} if $K$ is compact, $0\in\inte K$,
and at every point $x\in\partial K$ there exists a unique supporting
hyperplane to $K$. We say that $K\in\kon$ is \emph{strictly convex}
if $K$ is compact, $0\in\inte K$ and $\ext(K)=\partial K$. It is
a standard fact in convexity that $K$ is smooth if and only if its
polar $K^{\circ}$ is strictly convex.
\begin{thm}
Assume $K\in\kon$ is compact and $0\in\inte K$. Then $K^{\prime}$
is strictly convex.
\end{thm}
Ideologically, the theorem follows from the fact that for every $0<r<R<\infty$
the family 
\[
\left\{ P_{x}\cap B(0,R):\ r<\left|x\right|<R\right\} 
\]
 is ``uniformly convex'', i.e. has a uniform lower bound on its
modulus of convexity. It then follows that an arbitrary intersection
of such bodies will be strictly convex as well. In particular, since
for $R>0$ large enough we have $K^{\prime}=\bigcap_{x\in\partial K}\left(P_{x}\cap B(0,R)\right)$,
it follows that $K^{\prime}$ is strictly convex. Since filling in
the computational details is tedious and not very illuminating, we
will omit the formal proof. 

Instead, let us now fix a reciprocal body $K\in\rec$, and discuss
the class of ``pre-reciprocals'' $\left\{ A\in\kon:\ A'=K\right\} $.
It is obvious that such a pre-reciprocals are in general not unique.
For example, if $A\notin\rec$ then $A$ and $A''$ are two different
pre-reciprocals of $A'$. 

However, sometimes it is true that the pre-reciprocal is unique:
\begin{prop}
Let $K$ be a smooth convex body. Then there exists at most one body
$A$ such that $A'=K$. 
\end{prop}
\begin{proof}
Assume $A'=B'=K$. Then $\left(A^{\ind}\right)^{\circ}=\left(B^{\ind}\right)^{\circ}=K$,
which implies that $\conv\left(A^{\ind}\right)=\conv\left(B^{\ind}\right)=K^{\circ}$. 

Since $\conv\left(A^{\ind}\right)=K^{\circ}$ we have $A^{\ind}\supseteq\ext(K^{\circ})$.
Since $K$ is smooth its polar is strictly convex, so $A^{\ind}\supseteq\partial K^{\circ}$.
But $A^{\ind}$ is a star body, so we must have $A^{\ind}=K^{\circ}.$
Similarly $B^{\ind}=K^{\circ}$, and since $A^{\ind}=B^{\ind}$ we
conclude that $A=B$. 
\end{proof}
When $K$ is not smooth it may have many pre-reciprocals, but something
can still be said: The set \\
$\mathcal{D}(K)=\left\{ A\in\kon:\ A'=K\right\} $ is a convex subset
on $\kon$. 
\begin{thm}
\label{thm:pre-convex}
\begin{enumerate}
\item \label{enu:pre-convex-convex}Fix $K\in\kon$ such that $0\in\inte K$.
If $A,B\in\mathcal{D}(K)$ then $\lambda A+(1-\lambda)B\in\mathcal{D}(K)$
for all $0\le\lambda\le1$. 
\item \label{enu:pre-convex-largest}If $K\in\kon$ and $\mathcal{D}(K)\ne\emptyset$
then $K'$ is the largest body in $\mathcal{D}(K)$. 
\end{enumerate}
\end{thm}
For the proof we need the following lemma:
\begin{lem}
Let $X,Y\subseteq\RR^{n}$ be compact sets such that $\conv X=\conv Y=T$.
Then $\conv\left(X\cap Y\right)=\conv\left(X\cup Y\right)=T$. 
\end{lem}
\begin{proof}
For the union this is trivial: On the one $\conv\left(X\cup Y\right)\supseteq\conv X=T$.
On the other hand $X\cup Y\subseteq T$ and $T$ is convex, so $\conv\left(X\cup Y\right)\subseteq T$. 

For the intersection, the inclusion $\conv\left(X\cap Y\right)\subseteq T$
is again obvious. Conversely, since $\conv X=\conv Y=T$ it follows
that $X,Y\supseteq\ext(T)$, so $X\cap Y\supseteq\ext T$. It follows
from the Krein-{}-Milman theorem that $\conv\left(X\cap Y\right)\supseteq\conv\left(\ext T\right)=T$. 
\end{proof}
\begin{proof}[Proof of Theorem \ref{thm:pre-convex}]
For \eqref{enu:pre-convex-convex}, fix $A,B\in\mathcal{D}(K)$.
Since $A'=B'=K$ we have $\conv\left(A^{\ind}\right)=\conv\left(B^{\ind}\right)=K^{\circ}$. 

Write $C=\lambda A+(1-\lambda)B$. We have 
\[
r_{C^{\ind}}=h_{C}=\lambda h_{A}+(1-\lambda)h_{B}\le\max\left\{ h_{A},h_{B}\right\} =\max\left\{ r_{A^{\ind}},r_{B^{\ind}}\right\} =r_{A^{\ind}\cup B^{\ind}}.
\]
 Hence $C^{\ind}\subseteq A^{\ind}\cup B^{\ind}$, and similarly $C^{\ind}\supseteq A^{\ind}\cap B^{\ind}$.
It follows that 
\[
K^{\circ}=\conv\left(A^{\ind}\cap B^{\ind}\right)\subseteq\conv C^{\ind}\subseteq\conv\left(A^{\ind}\cup B^{\ind}\right)=K^{\circ},
\]
 so $C'=\left(C^{\ind}\right)^{\circ}=K^{\circ\circ}=K$. 

For \eqref{enu:pre-convex-largest}, $\mathcal{D}(K)\ne\emptyset$
exactly means that $K\in\rec$, so $K''=K$ and $K^{\prime}\in\mathcal{D}(K)$.
For any other $A\in\mathcal{D}(K)$ we have $A\subseteq A''=K^{\prime}$
so $K^{\prime}$ is indeed the largest body in $\mathcal{D}(K)$. 
\end{proof}
Note that Theorem \ref{thm:pre-convex} gives us a \emph{partition}
of the family of compact convex bodies in $\RR^{n}$ into \emph{convex}
sub-families, where $A$ and $B$ belong to the same sub-family if
and only if $A'=B'$. 

We conclude this section by turning our attention to Theorem \ref{thm:double-convex}.
For the full proof we will need some new ideas, presented in the next
section. But the ideas we developed so far suffice to give a simple
geometric proof of the theorem in some cases. We find it worthwhile,
as the proof of Section \ref{sec:main-proof} is not intuitive, and
the following proof shows why convexity of $K^{\ind}$ plays a role.
Let us show the following:
\begin{prop}
Assume that $K\in\rec$ is smooth. Then $K^{\ind}$ is convex.
\end{prop}
\begin{proof}
Assume by contradiction that $K^{\ind}$ is not convex. Then we can
choose a point $x\in\partial K^{\ind}\cap\inte\left(\conv K^{\ind}\right).$
Write $\hat{x}=\frac{x}{\left|x\right|}$. Since
\[
h_{K}\left(\hat{x}\right)=r_{K^{\ind}}\left(\hat{x}\right)=\left|x\right|,
\]
 we conclude that the hyperplane $H_{x}=\left\{ z:\ \left\langle z-x,x\right\rangle =0\right\} $
is a supporting hyperplane for $K$. Fix a point $y\in\partial K\cap H_{x}$. 

Since $y\in K$ we know that $[0,y]\subseteq K$, so $B_{y}=[0,y]^{\ind}\subseteq K^{\ind}$.
We claim that $B_{y}\cap\partial K^{\ind}=\left\{ x\right\} $. Indeed,
by elementary geometry (see Example \ref{exa:indicatrix-interval})
we know that $w\in\partial B_{y}$ if and only if $\measuredangle0wy=90^{\circ}$,
i.e. $\left\langle w,y-w\right\rangle =0$. This is also easy to check
algebraically. Since $y\in H_{x}$ we know that $\left\langle y-x,x\right\rangle =0$,
so $x\in B_{y}$.

Conversely, if $w\in B_{y}\cap\partial K^{\ind}$ then $y\in H_{w}=\left\{ z:\ \left\langle z-w,w\right\rangle =0\right\} $.
Again since $w\in\partial K^{\ind}$ we conclude that $H_{w}$ is
a supporting hyperplane for $K$. Since $H_{x}$ and $H_{w}$ are
two supporting hyperplanes passing through $y$, and since $K$ is
smooth, we must have $H_{x}=H_{w}$, so $x=w$. This proves the claim.

It follows in particular that $B_{y}\subseteq\inte\left(\conv K^{\ind}\right)$.
Since $B_{y}$ is compact and $\inte\left(\conv K^{\ind}\right)$
is open, it follows that $B_{z}\subseteq\inte\left(\conv K^{\ind}\right)$
for all $z$ close enough to $y$. In particular one may take $z=(1+\epsilon)y$
for a small enough $\epsilon>0$. Since $y\in\partial K$, $z\notin K$. 

Define $P=\text{conv}\left(K,z\right)=\conv\left(K\cup[0,z]\right)$.
Then 
\[
P^{\ind}=K^{\ind}\cup[0,z]^{\ind}=K^{\ind}\cup B_{z}\subseteq\conv\left(K^{\ind}\right).
\]
 Hence $\conv\left(P^{\ind}\right)=\conv\left(K^{\ind}\right)$, so
$P'=K'$. But then $K''=P''\supseteq P\supsetneq K$, so $K\notin\rec$.
\end{proof}

\section{\label{sec:main-proof}The spherical inversion and a proof of Theorem
\ref{thm:double-convex}}

The main goal of this section is to prove Theorem \ref{thm:double-convex}:
$K\in\rec$ if and only if $K^{\ind}$ is convex. For the proof we
will use the maps $\II$ and $\Phi$ from Definition \ref{def:indicatrix}.
We will use also the following well-known property of $\II$: 
\begin{fact}
Let $A\subseteq\RR^{n}$ be a sphere or a hyperplane. Then $\II\left(A\right)$
is a hyperplane if $0\in A$ , and a sphere if $0\notin A$. 
\end{fact}
It follows that if $B$ is any ball such that $0\in B$, then $\Phi(B)$
is either a ball (if $0\in\inte B$) or a half-space (if $0\in\partial B$). 

Since in this section we will compose many operations, it will be
more convenient to write them in function notation, where composition
is denoted by juxtaposition. For example, by $\circ\Phi\ind K$ we
mean $\left(\Phi\left(K^{\ind}\right)\right)^{\circ}$. In particular
$\circ\circ=\conv$, the (closed) convex hull operation. We have the
following relations between the different maps:
\begin{prop}
\label{prop:op-identities}If $K\in\kon$ then 
\begin{enumerate}
\item \label{enu:op-circ-star}$\circ\ind K=K^{\prime}$. 
\item \label{enu:op-phi-star}$\Phi\ind K=\circ K$. 
\item \label{enu:op-phi-circ}$\Phi\circ K=\ind K$.
\item \label{enu:op-star-circ}$\ind\circ K=\Phi K.$
\item \label{enu:op-interwind}$\left(\circ K\right)^{\prime}=\circ\Phi K$. 
\end{enumerate}
\end{prop}
\begin{proof}
Identity \eqref{enu:op-circ-star} is the same as Proposition \ref{prop:prime-as-polar}. 

For \eqref{enu:op-phi-star} we compare radial functions: 
\[
r_{\Phi\ind K}=\frac{1}{r_{\ind K}}=\frac{1}{h_{K}}=r_{\circ K}.
\]

\eqref{enu:op-phi-circ} follows from \eqref{enu:op-phi-star} by
applying $\Phi$ to both sides. 

For \eqref{enu:op-star-circ} we applying \eqref{enu:op-phi-circ}
to $\circ K$ instead of $K$ and obtain 
\[
\ind\circ K=\Phi\circ\circ K=\Phi K.
\]

\eqref{enu:op-interwind} is obtained from \eqref{enu:op-star-circ}
by taking polar of both sides and applying \eqref{enu:op-circ-star}.
\end{proof}
Note that Proposition \ref{prop:op-identities}\eqref{enu:op-phi-star}
provides a decomposition of the classical duality to a ``global''
part (the flower) and an ``essentially pointwise'' part (the map
$\Phi$). Also note that the identities \eqref{enu:op-phi-star} and
\eqref{enu:op-phi-circ} actually hold for all star bodies, since
$\ind A=\ind\conv A$ and $\circ A=\circ\conv A$. The convexity of
$K$ is crucial however for identity \eqref{enu:op-star-circ}, and
for general star bodies we only have $\ind\circ A=\Phi\conv A$. 

We will also need to know the following construction and its properties,
which may be of independent interest:
\begin{defn}
The \emph{spherical inner hull} of a convex body $K$ is defined by
\[
\inn K=\bigcup\left\{ B(x,\left|x\right|):\ B(x,\left|x\right|)\subseteq K\right\} .
\]
\end{defn}
\begin{prop}
\label{prop:convexity-preserving}Fix $K\in\kon$. Then 
\begin{enumerate}
\item \label{enu:phi-identity}We have the identity 
\begin{equation}
\inn K=\Phi\conv\Phi K=\Phi\circ\circ\Phi K\label{eq:phi-o-o-phi}
\end{equation}
\item \label{enu:phi-convex}$\inn K\in\kon$. In other words, \eqref{eq:phi-o-o-phi}
always defines a \emph{convex} subset of $K$. 
\item \label{enu:phi-charac}$\inn K$ is the largest star body $A\subseteq K$
such that $\Phi(A)$ is convex. In particular $\inn K=K$ if and only
if $\Phi(K)$ is convex.
\end{enumerate}
\end{prop}
\begin{proof}
For \eqref{enu:phi-identity} we should prove that $\Phi\conv\Phi K=\inn K$,
or equivalently that $\conv\Phi K=\Phi\inn K$. Since $\Phi$ is a
duality on star bodies we have 
\[
\Phi\inn K=\bigcap\left\{ \Phi B(x,\left|x\right|):\ \Phi B(x,\left|x\right|)\supseteq\Phi K\right\} .
\]

Since $\left\{ B(x,\left|x\right|):\ x\in\RR^{n}\right\} $ is exactly
the family of all balls having $0$ on their boundary, $\left\{ \Phi B(x,\left|x\right|):\ x\in\RR^{n}\right\} $
is the family of all affine half-spaces with $0$ in their interior.
Hence
\begin{align*}
\Phi\inn K & =\bigcap\left\{ H:\ \begin{array}{l}
H\text{ is a half-space}\\
0\in\inte H\text{ and }H\supseteq\Phi K
\end{array}\right\} =\conv\Phi K
\end{align*}
 which is what we wanted to prove.

To show \eqref{enu:phi-convex}, fix $x,y\in\inn K$ and $0<\lambda<1$.
We have $x\in B(a,\left|a\right|)\subseteq K$ and $y\in B(b,\left|b\right|)\subseteq K$
for some $a,b\in\RR^{n}$. Hence 
\begin{align*}
(1-\lambda)x+\lambda y & \in(1-\lambda)B(a,\left|a\right|)+\lambda B(b,\left|b\right|)\\
 & =B\left(\left(1-\lambda\right)a+\lambda b,(1-\lambda)\left|a\right|+\lambda\left|b\right|\right)\subseteq K.
\end{align*}
 Consider the ball $B=B\left(\left(1-\lambda\right)a+\lambda b,(1-\lambda)\left|a\right|+\lambda\left|b\right|\right)$.
Obviously $0\in B$. We know that $\Phi B$ is either a ball or a
half-space. In particular it is convex, so $\inn B=\Phi\conv\Phi B=\Phi\Phi B=B$.
Hence $(1-\lambda)x+\lambda y\in\inn B$ and we can find $c\in\RR^{n}$
such that 
\[
(1-\lambda)x+\lambda y\in B(c,\left|c\right|)\subseteq B\subseteq K.
\]
 It follows that $(1-\lambda)x+\lambda y\in\inn K$ and the proof
of \eqref{enu:phi-convex} is complete.

Finally we prove \eqref{enu:phi-charac}. The inequality $\inn K\subseteq K$
is obvious from the definition. Since 
\[
\Phi\left(\inn K\right)=\Phi\Phi\conv\Phi K=\conv\Phi K,
\]
 we see that $\Phi\left(\inn K\right)$ is convex. Next, we fix a
star body $A\subseteq K$ such that $\Phi(A)$ is convex. Then $\Phi(A)\supseteq\Phi(K)$,
and since $\Phi(A)$ is convex it follows that $\Phi\left(A\right)\supseteq\conv\Phi\left(K\right).$
Hence 
\[
A=\Phi\Phi A\subseteq\Phi\conv\Phi K=\inn K,
\]
 which is what we wanted to prove. 
\end{proof}
Now we can finally prove Theorem \ref{thm:double-convex}:
\begin{proof}[Proof of Theorem \ref{thm:double-convex}]
We start with the easy implication which does not require Proposition
\ref{prop:convexity-preserving}: Assume $\ind K$ is convex. Then
by Proposition \ref{prop:op-identities}\eqref{enu:op-star-circ}
we have $\ind\circ\ind K=\Phi\ind K$. Hence
\[
K''=\circ\ind\circ\ind K=\circ\Phi\ind K=\circ\circ K=K,
\]
 so $K\in\rec$. 

Conversely, assume that $K\in\rec$. Then $K''=K$, meaning that $\circ\ind\circ\ind K=K$.
As $\ind=\Phi\circ$ we have $\circ\Phi\circ\circ\Phi\circ K=K$.
Applying $\ind$ to both sides we get $\ind\circ\Phi\circ\circ\Phi\circ K=\ind K$.

Since $\circ K\in\kon$, Proposition \ref{prop:convexity-preserving}
implies that $\Phi\circ\circ\Phi\circ K\in\kon$. Hence by Proposition
\ref{prop:op-identities}\eqref{enu:op-star-circ} we have
\[
\ind\circ\Phi\circ\circ\Phi\circ K=\Phi\Phi\circ\circ\Phi\circ K=\circ\circ\Phi\circ K=\circ\circ\ind K.
\]
 We showed that $\ind K=\circ\circ\ind K=\conv\left(\ind K\right)$,
so $\ind K$ is convex. 
\end{proof}
As a corollary of the theorem we have the following result about projections:
\begin{prop}
\label{prop:proj-identity}Fix $K\in\rec$ and a subspace $E\subseteq\RR^{n}$.
Then $\left(\proj_{E}K\right)^{\prime}=\proj_{E}K'$. 
\end{prop}
The reciprocity on the left hand side is taken of course inside the
subspace $E$. This identity should be compared with the standard
identity 
\begin{equation}
\proj_{E}K^{\circ}=\left(K\cap E\right)^{\circ}\label{eq:polar-proj}
\end{equation}
 which holds for the polarity map.
\begin{proof}
Since $K\in\rec$ we know that $K^{\ind}$ is convex. By Proposition
\ref{prop:indicatrix-prop}\eqref{enu:ind-proj} and \eqref{eq:polar-proj}
we have
\[
\left(\proj_{E}K\right)^{\prime}=\left(\left(\proj_{E}K\right)^{\ind}\right)^{\circ}=\left(K^{\ind}\cap E\right)^{\circ}=\proj_{E}\left(K^{\ind}\right)^{\circ}=\proj_{E}K'.
\]
\end{proof}
\begin{rem}
Note that we only claimed the identity for reciprocal bodies. In fact,
if $\left(\proj_{E}K\right)^{\prime}=\proj_{E}K'$ for all $1$-dimensional
subspaces $E$, then $K\in\rec$. To see this, note $K''\in\rec$
and $K'=K'''$, so by Proposition \ref{prop:proj-identity} we have
\[
\left(\proj_{E}K\right)^{\prime}=\proj_{E}K'=\proj_{E}K'''=\left(\proj_{E}K''\right)^{\prime}.
\]
 Since every $1$-dimensional convex body is a reciprocal body we
deduce that $\proj_{E}K=\proj_{E}K''$ for all $1$-dimensional subspaces
$E$, so $K=K''\in\rec$. 
\end{rem}

\section{\label{sec:phi-properties}Structures on the class of flowers and
applications}

In general, the map $\Phi$ does not preserve convexity. We begin
this section by understanding when $\Phi(A)$ is convex:
\begin{prop}
\label{prop:phi-convex}Let $A$ be a star body. Then $\Phi(A)$ is
convex if and only if $A$ is a flower. 

Furthermore, the following are equivalent for a convex body $K\in\kon$:
\begin{enumerate}
\item \label{enu:phi-convex-basic}$\Phi(K)$ is convex. 
\item \label{enu:phi-convex-polar-rec}$K^{\circ}\in\rec$.
\item \label{enu:phi-convex-inner}$\inn K=K$.
\end{enumerate}
\end{prop}
\begin{proof}
For the first statement, note that if $A=T^{\ind}$ is a flower then
$\Phi(A)=\Phi\left(T^{\ind}\right)=T^{\circ}$ is convex (see Proposition
\ref{prop:op-identities}\eqref{enu:op-phi-star}). Conversely, Assume
$\Phi(A)=T$ is convex. Then $\Phi(A)=T=\Phi\left(\left(T^{\circ}\right)^{\ind}\right)$,
so $A=\left(T^{\circ}\right)^{\ind}$ is a flower.

For the second statement, the equivalence between \eqref{enu:phi-convex-basic}
and \eqref{enu:phi-convex-polar-rec} is exactly Theorem \ref{thm:double-convex}:
$K^{\circ}\in\rec$ if and only if $\left(K^{\circ}\right)^{\ind}=\Phi(K)$
is convex. The equivalence between \eqref{enu:phi-convex-basic} and
\eqref{enu:phi-convex-inner} was part of Proposition \ref{prop:convexity-preserving}.
\end{proof}
Of course, since $\Phi$ is an involution, the first half of Proposition
\ref{prop:phi-convex} means that the image $\Phi\left(\kon\right)$
is exactly the class of flowers. As for the second half, there are
examples of convex bodies $K\in\rec$ such that $K^{\circ}\notin\rec$,
so these are indeed different classes of convex bodies. 

We will now use Proposition \ref{prop:phi-convex} to study some structures
on the class of flowers. Recall that the radial sum $A\widetilde{+}B$
of two star bodies $A$ and $B$ is given by $r_{A\widetilde{+}B}=r_{A}+r_{B}$.
It is immediate that if $A$ and $B$ are flowers then so is $A\widetilde{+}B$,
and in fact 
\begin{equation}
K^{\ind}\widetilde{+}T^{\ind}=(K+T)^{\ind}.\label{eq:radial-minkowski}
\end{equation}

It is less obvious that the class of flowers is also closed under
the Minkowski addition:
\begin{prop}
\label{prop:ball-indicatrix}Let $B$ be any Euclidean ball with $0\in B$.
Then $B$ is a flower.
\end{prop}
\begin{proof}
We saw already that $\Phi\left(B\right)$ is always convex. Proposition
\ref{prop:phi-convex} finishes the proof. 
\end{proof}
\begin{thm}
\label{thm:flower-sum}Assume $A$ and $B$ are two flowers (which
are not necessarily convex). Then $A+B$ is also a flower, where $+$
is the usual Minkowski sum.
\end{thm}
\begin{proof}
Write $A=K^{\ind}$ and $B=T^{\ind}$ for $K,T\in\kon$. By Proposition
\ref{prop:repr-formulas} we have
\[
A=\bigcup_{x\in K}B_{x}\quad\text{ and }\quad B=\bigcup_{y\in T}B_{y}.
\]
 Hence 
\[
A+B=\bigcup_{\substack{x\in K\\
y\in T
}
}\left(B_{x}+B_{y}\right)=\bigcup_{\substack{x\in K\\
y\in T
}
}B\left(\frac{x+y}{2},\frac{\left|x\right|+\left|y\right|}{2}\right).
\]
 Since $0\in B\left(\frac{x+y}{2},\frac{\left|x\right|+\left|y\right|}{2}\right)$
the previous proposition implies that every such ball is a flower.
Since $A+B$ is a union of such balls, the claim follows (see Proposition
\ref{prop:indicatrix-prop}\eqref{enu:ind-union}). 
\end{proof}
\begin{rem}
Equation \eqref{eq:radial-minkowski} shows that the radial sum of
flowers corresponds to the Minkowski sum of convex bodies. Similarly,
Theorem \ref{thm:flower-sum} implies that the Minkowski sum of flowers
corresponds to an addition of convex bodies, defined implicitly by
\begin{equation}
K^{\ind}+T^{\ind}=\left(K\oplus T\right)^{\ind}.\label{eq:new-addition}
\end{equation}
 The addition $\oplus$ is associative, commutative, monotone and
has $\left\{ 0\right\} $ as its identity element. However, in general
it does not satisfy $K\oplus K=2K$, and in fact $K\oplus K$ is usually
not homothetic to $K$. The identity $K\oplus K=2K$ does hold if
$K$ is a reciprocal body. Moreover, if $K,T\in\rec$ then by Theorem
\ref{thm:double-convex} $K^{\ind}$ and $T^{\ind}$ are convex, so
$\left(K\oplus T\right)^{\ind}$ is convex and $K\oplus T\in\rec$
as well. In other words, $\rec$ is closed under $\oplus$. 
\end{rem}
Theorem \ref{thm:flower-sum} can be equivalently stated in the language
of the map $\Phi$:
\begin{cor}
\label{cor:inv-convex-sum}Let $A$ and $B$ be star bodies such that
$\Phi(A)$, $\Phi(B)$ are convex. Then $\Phi(A+B)$ is convex as
well. 
\end{cor}
There is also a similar statement for reciprocal bodies:
\begin{prop}
If $K,T\in\rec$ then $\left(K^{\circ}+T^{\circ}\right)^{\circ}\in\rec$.
\end{prop}
\begin{proof}
Write $A=K'$ and $B=T'$. Then $K=K''=A'=\left(A^{\ind}\right)^{\circ}$.
Since $A$ is a reciprocal body $A^{\ind}$ is convex, so $K^{\circ}=\left(A^{\ind}\right)^{\circ\circ}=A^{\ind}$.
In the same way we have $T^{\circ}=B^{\ind}$. Hence $K^{\circ}$
and $T^{\circ}$ are both flowers, so by the previous Proposition
$K^{\circ}+T^{\circ}$ is a flower. If we write $K^{\circ}+T^{\circ}=C^{\ind}$
then $\left(K^{\circ}+T^{\circ}\right)^{\circ}=C'\in\rec$.
\end{proof}
A similar phenomenon holds regarding sections and projections. If
$A\subseteq\RR^{n}$ is a flower and $E$ is a subspace of $\RR^{n}$
then we already saw in Proposition \ref{prop:indicatrix-prop}\eqref{enu:ind-proj}
that $A\cap E$ is a flower in $E$, and in fact $\left(\proj_{E}K\right)^{\ind}=K^{\ind}\cap E$.
It is less clear, but still true, that $\proj_{E}A$ is a flower as
well:
\begin{prop}
\label{prop:flower-proj}If $A\subseteq\RR^{n}$ is a flower and $E$
is a subspace of $\RR^{n}$, then $\proj_{E}A$ is a flower in $E$.
\end{prop}
\begin{proof}
If $A=K^{\ind}$ then $A=\bigcup_{x\in K}B_{x}$, and then 
\[
\proj_{E}A=\bigcup_{x\in K}\proj_{E}B_{x}.
\]
 Each projection $\proj_{E}B_{x}$ is a Euclidean ball in $E$ that
contains the origin, so by Proposition \ref{prop:ball-indicatrix}
is a flower. It follows that $\proj_{E}A$ is a flower as well.
\end{proof}
The last operation we would like to mention which preserves the class
of flowers is the convex hull:
\begin{prop}
\label{prop:flower-conv}If $A\subseteq\RR^{n}$ is a flower so is
$\conv\left(A\right)$, and in fact $\conv\left(K^{\ind}\right)=\left(K^{\prime\prime}\right)^{\ind}$. 
\end{prop}
\begin{proof}
Using the notation of Section \ref{sec:main-proof} we have $\left(K^{\prime\prime}\right)^{\ind}=\ind\circ\ind\circ\ind K$.
Since $\circ\ind K=K^{\prime}$ is obviously a reciprocal body, Theorem
\ref{thm:double-convex} implies that $\ind\circ\ind K$ is convex.
Hence by Proposition \ref{prop:op-identities} parts \eqref{enu:op-star-circ}
and \eqref{enu:op-phi-star} we have 
\[
\left(K^{\prime\prime}\right)^{\ind}=\ind\circ\left(\ind\circ\ind K\right)=\Phi\ind\circ\ind K=\circ\circ\ind K=\conv\left(K^{\ind}\right).
\]
\end{proof}
More structure on the class of flowers can be obtained by transferring
known results about the class $\kon$ of convex bodies. First let
us define the ``inverse flower'' operation:
\begin{defn}
The \emph{core }of a flower $A$ is defined by 
\[
A^{-\ind}=\left\{ x\in\RR^{n}:\ B_{x}\subseteq A\right\} .
\]
 
\end{defn}
In a recent paper (\cite{Zong2018}) Zong defined the core of a convex
body $T$ to be the Alexandrov body $A\left[r_{T}\right]$. This is
equivalent to our definition, though we apply it to flowers and not
to convex bodies. The core operation $-\ind$ is indeed the inverse
operation to $\ind$: For every $K\in\kon$ we have 
\[
\left(K^{\ind}\right)^{-\ind}=\left\{ x\in\RR^{n}:\ [0,x]^{\ind}\subseteq K^{\ind}\right\} =\left\{ x\in\RR^{n}:\ [0,x]\subseteq K\right\} =K.
\]
Equivalently, for every flower $A$ the set $K=A^{-\ind}$ is a convex
body and $K^{\ind}=A$. 

\selectlanguage{american}%
We already referred in the introduction to a characterization of the
polarity from \cite{Artstein-Avidan2008}. Essentially the same result
can also be formulated in terms of order-preserving transformations.
We say that a map $T:\kon\to\kon$ is order-preserving if $A\subseteq B$
if and only if $T(A)\subseteq T(B)$. Then the theorem states that
the only order-preserving bijections $T:\kon\to\kon$ are the (pointwise)
linear maps. From here we deduce:
\begin{prop}
Let $T:\FF\to\FF$ be an order-preserving bijection on the class of
flowers. Then there exists an invertible linear map $u:\RR^{n}\to\RR^{n}$
such that $T\left(A\right)=\left(uA^{-\ind}\right)^{\ind}$. 
\end{prop}
\begin{proof}
Define $S:\kon\to\kon$ by $S(K)=\left(T\left(K^{\ind}\right)\right)^{-\ind}$.
Then $S$ is easily seen to be an order preserving bijection on the
class $\kon$. Hence by the above-mentioned result from \cite{Artstein-Avidan2008}
there exists a linear map $u:\RR^{n}\to\RR^{n}$ such that $S(K)=uK$.
It follows that $T\left(A\right)=\left(uA^{-\ind}\right)^{\ind}$
like we wanted.
\end{proof}
Note that even though $S$ in the proof above is linear, the map $T$
is in general not even a pointwise map. In fact, it can be quite complicated
– it does not preserve convexity for example. 

With the same proof one may also characterize all dualities on flowers,
i.e. all order-reversing involutions:
\begin{prop}
Let $T:\FF\to\FF$ be an order-reversing involution on the class of
flowers. Then there exists an invertible symmetric linear map $u:\RR^{n}\to\RR^{n}$
such that $T\left(A\right)=\left(\left(uA^{-\ind}\right)^{\circ}\right)^{\ind}$. 
\end{prop}
\selectlanguage{english}%
We conclude this section with a nice example. Let $B$ be any Euclidean
ball with $0\in B$. By Proposition \ref{prop:ball-indicatrix} we
know that $B=K^{\ind}$ for some body $K$. What is $K$? It turns
out that $K$ is an ellipsoid. As $(uK)^{\ind}=u\left(K^{\ind}\right)$
for every orthogonal matrix $u$, the body $K$ is clearly a body
of revolution. Hence the problem is actually $2$-dimensional and
we may assume that $n=2$. 

Up to rotation, every ellipse has the form
\[
E=\left\{ \frac{(x-x_{0})^{2}}{a^{2}}+\frac{\left(y-y_{0}\right)^{2}}{b^{2}}\le1\right\} \subseteq\RR^{2}
\]
 for $a>b>0$. Recall that $(x_{0},y_{0})$ is the center of the ellipse.
If we write $c=\sqrt{a^{2}-b^{2}}$ then $p_{1}=(x_{0}+c,y_{0})$
and $p_{2}=\left(x_{0}-c,y_{0}\right)$ are the foci of $E$, and
\[
E=\left\{ q\in\RR^{2}:\ \left|q-p_{1}\right|+\left|q-p_{2}\right|=2a\right\} .
\]
 The number $e=\sqrt{1-\frac{b^{2}}{a^{2}}}$ is the eccentricity
of $e$. Obviously every ellipse in $\RR^{2}$ is uniquely determined
by its center, its eccentricity and one of its focus points. We then
have:
\begin{prop}
Let $E\subseteq\RR^{2}$ be an ellipse with center at $p\in\RR^{2}$,
one focus point at $0$ and eccentricity $e$. Then:
\begin{enumerate}
\item \label{enu:ellipse-flower}$E^{\ind}$ is a ball with center $p$
and radius $\frac{\left|p\right|}{e}$. 
\item \label{enu:ellipse-prime}$E^{\prime}$ is an ellipse with center
$\widetilde{p}=-\frac{e^{2}}{1-e^{2}}\cdot\frac{p}{\left|p\right|^{2}}$,
a focus point at $0$ and eccentricity $e$. 
\end{enumerate}
\end{prop}
\begin{proof}
By rotating and scaling it is enough to assume that the center of
the ellipse is at $p=(1,0)$. We then have 
\[
E=\left\{ (x,y):\ \frac{(x-1)^{2}}{a^{2}}+\frac{y^{2}}{a^{2}-1}\le1\right\} ,
\]
 where $a=\frac{1}{e}>1$. To prove \eqref{enu:ellipse-flower}, consider
the centered ellipse $\widetilde{E}=E-p$. For such ellipses it is
well-known that $h_{\widetilde{E}}(x,y)=\sqrt{a^{2}x^{2}+(a^{2}-1)y^{2}}$,
and then
\[
h_{E}(x,y)=h_{\widetilde{E}}(x,y)+h_{\left\{ (1,0)\right\} }(x,y)=\sqrt{a^{2}x^{2}+(a^{2}-1)y^{2}}+x
\]

(note that we consider $h_{\widetilde{E}}$ and $h_{E}$ not as functions
on $S^{n-1}$, but as $1$-homogeneous functions defined on all of
$\RR^{n}$). Therefore 
\begin{align*}
E^{\ind} & =\left\{ (x,y):\ \left|(x,y)\right|\le r_{E^{\ind}}\left(\frac{(x,y)}{\left|(x,y)\right|}\right)\right\} =\left\{ (x,y):\ h_{E}(x,y)\ge\left|(x,y)\right|^{2}\right\} \\
 & =\left\{ (x,y):\ \sqrt{a^{2}x^{2}+(a^{2}-1)y^{2}}+x\ge x^{2}+y^{2}\right\} \\
 & =\left\{ (x,y):\ \left(x-1\right)^{2}+y^{2}\le a^{2}\right\} ,
\end{align*}
 where the last equality follows from simple algebraic manipulations.
We see that $E^{\ind}$ is indeed a ball with center $p=(1,0)$ and
radius $\frac{\left|p\right|}{e}=a$. 

To prove \eqref{enu:ellipse-prime}, recall that $E'=\left(E^{\ind}\right)^{\circ}$.
Like before, if $\widetilde{B}=B\left((0,0),a\right)$ is the centered
ball then 
\[
h_{E^{\ind}}(x,y)=h_{\widetilde{B}}(x,y)+h_{\left\{ (1,0)\right\} }(x,y)=a\sqrt{x^{2}+y^{2}}+x.
\]
 Hence 
\begin{align*}
E^{\prime}=\left(E^{\ind}\right)^{\circ} & =\left\{ (x,y):\ h_{E^{\ind}}(x,y)\le1\right\} \\
 & =\left\{ (x,y):\ a\sqrt{x^{2}+y^{2}}+x\le1\right\} .
\end{align*}
 Again, some algebraic manipulations will give us the (unpleasant)
canonical form 
\[
E^{\prime}=\left\{ \left(x,y\right):\ \frac{\left(x+\frac{1}{a^{2}-1}\right)^{2}}{\left(\frac{a}{a^{2}-1}\right)^{2}}+\frac{y^{2}}{\frac{1}{a^{2}-1}}\le1\right\} .
\]

Hence the center of $E^{\prime}$ is indeed at $\left(-\frac{1}{a^{2}-1},0\right)=\left(-\frac{e^{2}}{1-e^{2}},0\right)=-\frac{e^{2}}{1-e^{2}}\cdot\frac{p}{\left|p\right|^{2}}$.
The distance from the center to the foci is
\[
\sqrt{\left(\frac{a}{a^{2}-1}\right)^{2}-\frac{1}{a^{2}-1}}=\frac{1}{a^{2}-1}=\frac{e^{2}}{1-e^{2}},
\]
 so one of the focus points is indeed the origin. Finally, the eccentricity
of $E^{\prime}$ is indeed
\[
\sqrt{1-\frac{\frac{1}{a^{2}-1}}{\left(\frac{a}{a^{2}-1}\right)^{2}}}=\frac{1}{a}=e.
\]
\end{proof}
This proposition also gives a nice example of the addition $\oplus$
defined in \eqref{eq:new-addition}: For every $x_{1},x_{2},\ldots,x_{m}\in\RR^{n}$
the body $\bigoplus_{i=1}^{m}[0,x_{i}]$ is an ellipsoid. Indeed,
we have 
\[
\left(\bigoplus_{i=1}^{m}[0,x_{i}]\right)^{\ind}=\sum_{i=1}^{m}[0,x_{i}]^{\ind}=\sum_{i=1}^{m}B_{x_{i}}
\]
 which is a non-centered Euclidean ball, so by the last computation
$\bigoplus_{i=1}^{m}[0,x_{i}]$ is an ellipsoid of revolution with
one focus point at $0$. 

\section{\label{sec:geometric-Ineq}Geometric Inequalities}

In this final section we discuss several inequalities involving flowers
and reciprocal bodies. We begin by showing that the various operations
constructed in this paper are convex maps. A theorem of Firey (\cite{Firey1961})
implies that the polarity map $\circ:\kon\to\kon$ is convex: For
every $K,T\in\kon$ and every $0\le\lambda\le1$ one has 
\[
\left((1-\lambda)K+\lambda T\right)^{\circ}\subseteq(1-\lambda)K^{\circ}+\lambda T^{\circ}.
\]

We then have:
\begin{thm}
The map $\ind:\kon\to\FF$ is convex. The map $\Phi$ is convex when
applied to arbitrary star bodies.
\end{thm}
\begin{proof}
For any two star bodies $A$ and $B$ we have $r_{A+B}\ge r_{A}+r_{B}$.
Hence for $K,T\in\kon$ and $0\le\lambda\le1$ we have
\begin{align*}
r_{\left((1-\lambda)K+\lambda T\right)^{\ind}} & =h_{\left(1-\lambda\right)K+\lambda T}=(1-\lambda)h_{K}+\lambda h_{T}\\
 & =(1-\lambda)r_{K^{\ind}}+\lambda r_{T^{\ind}}\le r_{(1-\lambda)K^{\ind}+\lambda T^{\ind}}.
\end{align*}
 It follows that $\left((1-\lambda)K+\lambda T\right)^{\ind}\subseteq(1-\lambda)K^{\ind}+\lambda T^{\ind}$
so $\ind$ is convex.

For the convexity of $\Phi$ fix star bodies $A$ and $B$ and $0\le\lambda\le1$,
and note that
\begin{align*}
r_{\Phi\left((1-\lambda)A+\lambda B\right)} & =\frac{1}{r_{(1-\lambda)A+\lambda B}}\le\frac{1}{(1-\lambda)r_{A}+\lambda r_{B}}\stackrel{\left(\ast\right)}{\le}\frac{1-\lambda}{r_{A}}+\frac{\lambda}{r_{B}}\\
 & =(1-\lambda)r_{\Phi(A)}+\lambda r_{\Phi(B)}\le r_{(1-\lambda)\Phi(A)+\lambda\Phi(B)},
\end{align*}
where the inequality $\left(\ast\right)$ is the convexity of the
map $x\mapsto\frac{1}{x}$ on $(0,\infty)$. 
\end{proof}
Convexity of the reciprocal map is more delicate. For general convex
bodies $K,T\in\kon$ the inequality
\[
\left((1-\lambda)K+\lambda T\right)^{\prime}\subseteq(1-\lambda)K^{\prime}+\lambda T^{\prime}
\]
 is \emph{false}. It becomes true if we further assume that $K$ and
$T$ are reciprocal bodies: If $K\in\rec$ then $K^{\ind}$ is convex,
which means that $\frac{1}{r_{K^{\ind}}}=\frac{1}{h_{K}}$ is the
support function of a convex body. Hence $h_{K^{\prime}}=h_{A\left[1/h_{K}\right]}=\frac{1}{h_{K}}$
and similarly $h_{T^{\prime}}=\frac{1}{h_{T}}$. Therefore we indeed
have 
\begin{align*}
h_{\left((1-\lambda)K+\lambda T\right)^{\prime}} & \le\frac{1}{h_{(1-\lambda)K+\lambda T}}=\frac{1}{(1-\lambda)h_{K}+\lambda h_{T}}\le\frac{1-\lambda}{h_{K}}+\frac{\lambda}{h_{T}}\\
 & =(1-\lambda)h_{K^{\prime}}+\lambda h_{T^{\prime}}=h_{(1-\lambda)K^{\prime}+\lambda T^{\prime}}.
\end{align*}
 However, one cannot really say that $\prime$ is a convex map on
$\rec$ in the standard sense, since the class $\rec$ is not closed
with respect to the Minkowski addition. In Equation \eqref{eq:new-addition}
of the previous section we defined a new addition $\oplus$ which
does preserve the class $\rec$, and the following holds:
\begin{prop}
The reciprocal map $\prime:\rec\to\rec$ is convex with respect to
the addition $\oplus$.
\end{prop}
\begin{proof}
For every $K,T\in\rec$ we have 
\[
h_{K\oplus T}=r_{\left(K\oplus T\right)^{\ind}}=r_{K^{\ind}+T^{\ind}}\ge r_{K^{\ind}}+r_{T^{\ind}}=h_{K}+h_{T}=h_{K+T},
\]
 so $K\oplus T\supseteq K+T$. Hence by the convexity of $\circ$
we have 
\begin{align*}
\left((1-\lambda)K\oplus\lambda T\right)^{\prime} & =\left[\left((1-\lambda)K\oplus\lambda T\right)^{\ind}\right]^{\circ}=\left(\left(1-\lambda\right)K^{\ind}+\lambda T^{\ind}\right)^{\circ}\\
 & \subseteq(1-\lambda)\left(K^{\ind}\right)^{\circ}+\lambda\left(T^{\ind}\right)^{\circ}\subseteq(1-\lambda)K^{\prime}\oplus\lambda T^{\prime}.
\end{align*}
\end{proof}
We now turn our attention to numerical inequalities involving flowers.
To each body $K$ we can associate a new numerical parameter which
is $\left|K^{\ind}\right|$, the volume of the flower of $K$. For
example, it was explained in Remark \ref{rem:voronoi} why this volume
is important in stochastic geometry. We then have the following reverse
Brunn-Minkowski inequality:
\begin{prop}
For every $K,T\in\kon$ one has $\left|\left(K+T\right)^{\ind}\right|^{\frac{1}{n}}\le\left|K^{\ind}\right|^{\frac{1}{n}}+\left|T^{\ind}\right|^{\frac{1}{n}}$
.
\end{prop}
\begin{proof}
Recall that for every star body $A$ in $\RR^{n}$ we may integrate
by polar coordinates and deduce that $\left|A\right|=\left|B_{2}^{n}\right|\cdot\int_{S^{n-1}}r_{A}(\theta)^{n}\dd\sigma(\theta)$.
Here $\sigma$ denotes the uniform probability measure on the sphere.
It follows that for every $K\in\kon$ we have 
\begin{equation}
\left|K^{\ind}\right|=\left|B_{2}^{n}\right|\cdot\int_{S^{n-1}}h_{K}(\theta)^{n}\dd\sigma(\theta).\label{eq:flower-vol}
\end{equation}
In other words, $\left|K^{\ind}\right|^{\frac{1}{n}}$ is proportional
to $\left\Vert h_{K}\right\Vert _{L^{n}(S^{n-1})}$, where $L^{n}(S^{n-1})$
is the relevant $L^{p}$ space. Therefore the required inequality
is nothing more than Minkowski's inequality (the triangle inequality
for $L^{p}$-norms, in our case for $p=n$). 
\end{proof}
Similarly, we have an analogue of Minkowski's theorem on the polynomiality
of volume. Recall that for every fixed convex bodies $K_{1},K_{2},\ldots,K_{m}$
we have
\[
\left|\lambda_{1}K_{1}+\lambda_{2}K_{2}+\cdots+\lambda_{m}K_{m}\right|=\sum_{i_{1},i_{2},\ldots,i_{n}=1}^{m}V(K_{i_{1}},K_{i_{2}},\ldots,K_{i_{n}})\cdot\lambda_{i_{1}}\lambda_{i_{2}}\cdots\lambda_{i_{n}},
\]
Where we take the coefficients $V(K_{i_{1}},K_{i_{2}},\ldots,K_{i_{n}})$
to be symmetric with respect to a permutation of the arguments. The
number $V(K_{1},K_{2},\ldots,K_{n})$ is called the mixed volume of
$K_{1},K_{2},\ldots,K_{n}$ and is fundamental to convex geometry.
We then have:
\begin{prop}
Fix $K_{1},K_{2},\ldots,K_{m}\in\kon$. Then 
\[
\left|\left(\lambda_{1}K_{1}+\lambda_{2}K_{2}+\cdots+\lambda_{m}K_{m}\right)^{\ind}\right|=\sum_{i_{1},i_{2},\ldots,i_{n}=1}^{m}V^{\ind}(K_{i_{1}},K_{i_{2}},\ldots,K_{i_{n}})\cdot\lambda_{i_{1}}\lambda_{i_{2}}\cdots\lambda_{i_{n}},
\]
 where the coefficients are given by 
\begin{equation}
V^{\ind}(K_{1},K_{2},\ldots,K_{n})=\left|B_{2}^{n}\right|\cdot\int_{S^{n-1}}h_{K_{1}}(\theta)h_{K_{2}}(\theta)\cdots h_{K_{n}}(\theta)\dd\sigma(\theta).\label{eq:flower-mixed}
\end{equation}
\end{prop}
The proof is immediate from formula \eqref{eq:flower-vol}. Moreover,
the new $\ind$-mixed volumes satisfy a reverse (elliptic) Alexandrov-Fenchel
type inequality:
\begin{prop}
\label{prop:flower-AF}For every $K_{1},K_{2},\ldots,K_{n}\in\kon$
we have
\[
V^{\ind}(K_{1},K_{2},K_{3},\ldots,K_{n})^{2}\le V^{\ind}\left(K_{1},K_{1},K_{3},\ldots K_{n}\right)\cdot V^{\ind}\left(K_{2},K_{2},K_{3},\ldots,K_{n}\right),
\]
as well as
\[
V^{\ind}(K_{1},K_{2},\ldots,K_{n})\le\left(\prod_{i=1}^{n}\left|K_{i}^{\ind}\right|\right)^{\frac{1}{n}}.
\]
\end{prop}
\begin{proof}
Apply Hölder's inequality to formula \eqref{eq:flower-mixed}. 
\end{proof}
These results and their proofs are very closely related to the dual
Brunn–Minkowski theory which was developed by Lutwak in \cite{Lutwak1975}. 

Next we would like to compare the $\ind$-mixed volume $V^{\ind}(K_{1},K_{2},\ldots,K_{n})$
with the classical mixed volume $V(K_{1},K_{2},\ldots,K_{n})$. Since
$\left|T^{\ind}\right|\ge\left|T\right|$ for every $T\in\kon$, one
may conjecture that $V^{\ind}(K_{1},K_{2},\ldots,K_{n})\ge V(K_{1},K_{2},\ldots,K_{n})$.
This is not true however, as the next example shows:
\begin{example}
Let $\left\{ e_{1},e_{2}\right\} $ be the standard basis of . Define
$K=[-e_{1},e_{1}]$ and $T=[-e_{2},e_{2}]$. Then $\left|\lambda K+\mu T\right|=4\lambda\mu$
which implies that $V(K,T)=2$. 

On the other hand by Formula \eqref{eq:flower-mixed} we have 
\[
V^{\ind}(K,T)=\left|B_{2}^{2}\right|\cdot\int_{S^{1}}h_{K}(\theta)h_{T}(\theta)\dd\sigma(\theta)=\pi\cdot\frac{1}{2\pi}\int_{0}^{2\pi}\left|\cos\theta\right|\left|\sin\theta\right|\dd\theta=1,
\]
 so $V(K,T)>V^{\ind}(K,T)$.
\end{example}
However, in one case we can compare the $\ind$-mixed volume with
the classical one. Recall that for $K\in\kon$ and $0\le i\le n$
the $i$'th quermassintegral of $K$ is defined by
\[
W_{i}(K)=V\left(\underbrace{K,K,\ldots,K}_{n-i\text{ times}},\underbrace{B_{2}^{n},B_{2}^{n},\ldots,B_{2}^{n}}_{i\text{ times}}\right).
\]
Kubota's formula then states that 
\[
W_{n-i}(K)=\frac{\left|B_{2}^{n}\right|}{\left|B_{2}^{i}\right|}\cdot\int_{G(n,i)}\left|\proj_{E}K\right|\dd\mu(E),
\]
 where $G(n,i)$ is the set of all $i$-dimensional linear subspaces
of $\RR^{n}$, and $\mu$ is the Haar probability measure on $G(n,i)$. 

We define the $\ind$-quermassintegrals in the obvious way as $W_{i}^{\ind}(K)=V^{\ind}(\underbrace{K,\ldots,K}_{n-i},\underbrace{B_{2}^{n},\ldots,B_{2}^{n}}_{i})$.
We then have a Kubota–type formula:
\begin{thm}
For every $K\in\kon$ and every $0\le i\le n$ we have 
\[
W_{n-i}^{\ind}(K)=\frac{\left|B_{2}^{n}\right|}{\left|B_{2}^{i}\right|}\cdot\int_{G(n,i)}\left|\left(\proj_{E}K\right)^{\ind}\right|\dd\mu(E),
\]
 where $\mu$ is the Haar probability measure on $G(n,i)$ and the
flower map $\ind$ on the right hand side is taken inside the subspace
$E$. 
\end{thm}
\begin{proof}
If $T\subseteq\RR^{m}$ then integrating in polar coordinates we have
$\left|T\right|=\left|B_{2}^{m}\right|\cdot\int_{S^{m-1}}r_{T}(\theta)^{m}\dd\sigma_{m}(\theta)$,
where $\sigma_{m}$ denotes the Haar probability measure on $S^{m-1}$.
Therefore 
\begin{align*}
\int_{G(n,i)}\left|\left(\proj_{E}K\right)^{\ind}\right|\dd\mu(E) & =\int_{G(n,i)}\left|K^{\ind}\cap E\right|\dd\mu(E)=\left|B_{2}^{i}\right|\int_{G(n,i)}\int_{S_{E}}r_{K^{\ind}}(\theta)^{i}\dd\sigma_{E}(\theta)\dd\mu(E)\\
 & =\left|B_{2}^{i}\right|\int_{S^{n-1}}r_{K^{\ind}}(\theta)^{i}\dd\sigma_{n}(\theta)=\left|B_{2}^{i}\right|\int_{S^{n-1}}h_{K}(\theta)^{i}\dd\sigma_{n}(\theta)\\
 & =\frac{\left|B_{2}^{i}\right|}{\left|B_{2}^{n}\right|}W_{n-i}^{\ind}(K).
\end{align*}
\end{proof}
And as a corollary we obtain:
\begin{cor}
For every $K\in\kon$ and $0\le i\le n$ we have $W_{i}^{\ind}(K)\ge W_{i}(K)$. 
\end{cor}
\begin{proof}
We have
\[
W_{n-i}(K)=\frac{\left|B_{2}^{n}\right|}{\left|B_{2}^{i}\right|}\cdot\int_{G(n,i)}\left|\proj_{E}K\right|\dd\mu(E)\le\frac{\left|B_{2}^{n}\right|}{\left|B_{2}^{i}\right|}\cdot\int_{G(n,i)}\left|\left(\proj_{E}K\right)^{\ind}\right|\dd\mu(E)=W_{n-i}^{\ind}(K).
\]
\end{proof}
It is well known that $W_{n-1}(K)$ is (up to normalization) the mean
width of $K$. Hence from formula \eqref{eq:flower-mixed} we immediately
have $W_{n-1}^{\ind}(K)=W_{n-1}(K)$. The Alexandrov-Fenchel inequality
and its flower version from Proposition \ref{prop:flower-AF} then
imply that 
\begin{align*}
\left(\frac{\left|K\right|}{\left|B_{2}^{n}\right|}\right)^{\frac{1}{n}} & \le\left(\frac{W_{1}(K)}{\left|B_{2}^{n}\right|}\right)^{\frac{1}{n-1}}\le\cdots\le\left(\frac{W_{n-2}(K)}{\left|B_{2}^{n}\right|}\right)^{\frac{1}{2}}\le\frac{W_{n-1}(K)}{\left|B_{2}^{n}\right|}\\
 & =\frac{W_{n-1}^{\ind}(K)}{\left|B_{2}^{n}\right|}\le\left(\frac{W_{n-2}^{\ind}(K)}{\left|B_{2}^{n}\right|}\right)^{\frac{1}{2}}\le\cdots\le\left(\frac{W_{1}^{\ind}(K)}{\left|B_{2}^{n}\right|}\right)^{\frac{1}{n-1}}\le\left(\frac{\left|K^{\ind}\right|}{\left|B_{2}^{n}\right|}\right)^{\frac{1}{n}}
\end{align*}
 which gives another proof of the relation $W_{i}^{\ind}(K)\ge W_{i}(K)$. 

We conclude this paper with a remark regarding the distance of flowers
and reciprocal bodies to the Euclidean ball. We restrict ourselves
to bodies which are compact and contain $0$ at their interior. The
\emph{geometric distance}\textbf{\emph{ }}between such bodies $K$
and $T$ is 
\[
d(K,T)=\inf\left\{ \frac{b}{a}:\ aK\subseteq T\subseteq bK\right\} .
\]
 Recall that a body $K$ is centrally symmetric if $K=-K$. 
\begin{prop}
\begin{enumerate}
\item If a flower $A$ is centrally symmetric and convex, then $d\left(A,B_{2}^{n}\right)\le2$. 
\item If $K\in\rec$ is centrally symmetric, then $d(K,B_{2}^{n})\le2$. 
\end{enumerate}
\end{prop}
\begin{proof}
To prove the first assertion, write $A=K^{\ind}$ and let $R=\max_{x\in K}\left|x\right|$.
Since $K\subseteq R\cdot B_{2}^{n}$ we have $A\subseteq R\cdot B_{2}^{n}$. 

On the other hand, fix $x\in K$ with $\left|x\right|=R$ and note
that $B_{x}=[0,x]^{\ind}\subseteq K^{\ind}=A$. Since $K$ is centrally
symmetric we also have $-x\in K$, so $B_{-x}\subseteq A$. Hence
\[
\frac{R}{2}\cdot B_{2}^{n}\subseteq\conv\left(B_{x}\cup B_{-x}\right)\subseteq A,
\]
 so $d\left(A,B_{2}^{n}\right)\le2$.

For the second assertion, fix a centrally symmetric reciprocal body
$K$ and define $T=K'$. Then $K=T^{\prime}=\left(T^{\ind}\right)^{\circ}$.
Since $T$ is a reciprocal body $T^{\ind}$ is convex, so $d\left(T^{\ind},B_{2}^{n}\right)=d(K^{\circ},B_{2}^{n})\le2$.
Since polarity preserves the geometric distance we also have $d\left(K,B_{2}^{n}\right)\le2$. 
\end{proof}
Note that these results are false if $K$ is not centrally symmetric.
For example, we already saw in Proposition \ref{prop:ball-indicatrix}
that if $B$ is any ball with $0\in B$ then $B$ is a flower. But
if $0$ is close to $\partial B$ then $d\left(B,B_{2}^{n}\right)$
can be made arbitrarily large.

\bibliographystyle{plain}
\bibliography{../../citations/library}

\begin{thebibliography}{10}

\bibitem{Artstein-Avidan2008}
Shiri Artstein-Avidan and Vitali Milman.
\newblock {The concept of duality for measure projections of convex bodies}.
\newblock {\em Journal of Functional Analysis}, 254(10):2648--2666, may 2008.

\bibitem{Boroczky2012}
K{\'{a}}roly~J. B{\"{o}}r{\"{o}}czky, Erwin Lutwak, Deane Yang, and Gaoyong
  Zhang.
\newblock {The log-Brunn-Minkowski inequality}.
\newblock {\em Advances in Mathematics}, 231(3-4):1974--1997, oct 2012.

\bibitem{Boroczky2008}
K{\'{a}}roly~J. B{\"{o}}r{\"{o}}czky and Rolf Schneider.
\newblock {A characterization of the duality mapping for convex bodies}.
\newblock {\em Geometric and Functional Analysis}, 18(3):657--667, aug 2008.

\bibitem{Firey1961}
William~J. Firey.
\newblock {Polar means of convex bodies and a dual to the Brunn-Minkowski
  theorem}.
\newblock {\em Canadian Journal of Mathematics}, 13:444--453, 1961.

\bibitem{Gruber1992}
Peter~M. Gruber.
\newblock {The endomorphisms of the lattice of norms in finite dimensions}.
\newblock {\em Abhandlungen aus dem Mathematischen Seminar der
  Universit{\"{a}}t Hamburg}, 62(1):179--189, 1992.

\bibitem{Lutwak1975}
Erwin Lutwak.
\newblock {Dual mixed volumes}.
\newblock {\em Pacific Journal of Mathematics}, 58(2):531--538, 1975.

\bibitem{Milman2016a}
Vitali Milman and Liran Rotem.
\newblock {Non-standard constructions in convex geometry; geometric means of
  convex bodies}.
\newblock In Eric Carlen, Mokshay Madiman, and Elisabeth Werner, editors, {\em
  Convexity and Concentration}, volume 161 of {\em The IMA Volumes in
  Mathematics and its Applications}, pages 361--390. Springer, New York, NY,
  2017.

\bibitem{Milman2017a}
Vitali Milman and Liran Rotem.
\newblock {Powers and logarithms of convex bodies}.
\newblock {\em Comptes Rendus Mathematique}, 355(9):981--986, sep 2017.

\bibitem{Milman2018}
Vitali Milman and Liran Rotem.
\newblock {Weighted geometric means of convex bodies}.
\newblock In Peter Kuchment and Evgeny Semenov, editors, {\em Selim Krein
  Centennial}, Contemporary Mathematics. AMS, 2019.

\bibitem{Molchanov2014}
Ilya Molchanov.
\newblock {Continued fractions built from convex sets and convex functions}.
\newblock {\em Communications in Contemporary Mathematics}, 17(05):1550003, oct
  2015.

\bibitem{Moszynska1999}
Maria Moszy{\'{n}}ska.
\newblock {Quotient Star Bodies, Intersection Bodies, and Star Duality}.
\newblock {\em Journal of Mathematical Analysis and Applications},
  232(1):45--60, apr 1999.

\bibitem{Schneider2013}
Rolf Schneider.
\newblock {\em {Convex Bodies: The Brunn-Minkowski Theory}}.
\newblock Encyclopedia of Mathematics and its Applications. Cambridge
  University Press, Cambridge, second edition, 2014.

\bibitem{Spodarev2013}
Evgeny Spodarev, editor.
\newblock {\em {Stochastic Geometry, Spatial Statistics and Random Fields}},
  volume 2068 of {\em Lecture Notes in Mathematics}.
\newblock Springer, Berlin, Heidelberg, 2013.

\bibitem{Zong2018}
Chuanming Zong.
\newblock {A Computer Approach to Determine the Densest Translative Tetrahedron
  Packings}.
\newblock {\em arXiv:1805.02222}, may 2018.

\end{thebibliography}

\end{document}